\documentclass[11pt]{amsart}

\pdfoutput=1

\usepackage{graphicx}

\usepackage[colorlinks]{hyperref}

\usepackage{epstopdf}

\usepackage{amsmath,amscd,color}

\newtheorem{thm}{Theorem}[section]
\newtheorem{cor}[thm]{Corollary}
\newtheorem{lem}[thm]{Lemma}
\newtheorem{prop}[thm]{Proposition}

\theoremstyle{definition}
\newtheorem{defn}[thm]{Definition}

\theoremstyle{remark}
\newtheorem{rem}[thm]{Remark}
\numberwithin{equation}{section}

\newcommand{\R}{\mathbb R}
\newcommand{\Z}{\mathbb Z}

\newcommand{\D}{\mathbb D}

\newcommand{\F}{\mathcal{F}}

\newcommand{\PP}{{\mathbb P}}
\newcommand{\ra}{\rightarrow}

\newcommand{\g}{\gamma}
\newcommand{\G}{\Gamma}

\newcommand{\OO}{{\mathcal O}}

\newcommand{\Sb}{{\mathbb S}}

\newcommand{\sz}{SL(2,\Z)}
\newcommand{\h}{\mathcal H}

\begin{document}


\title[Boundary maps for Fuchsian groups]{Structure of attractors for boundary maps associated to Fuchsian groups}

\dedicatory{Dedicated to the memory of Roy Adler}

\author{Svetlana Katok}
\address{Department of Mathematics, The Pennsylvania State University, University Park, PA 16802} 
\email{sxk37@psu.edu}

\author{Ilie Ugarcovici}
\address{Department of Mathematical Sciences, DePaul University, Chicago, IL 60614}
\email{iugarcov@depaul.edu}
\thanks{The second author is partially supported by a Simons Foundation Collaboration Grant}

\date{September 30, 2016; revised May 1, 2017; accepted for publication in Geometriae Dedicata}

\begin{abstract}
We study dynamical properties of generalized Bowen-Series boundary maps associated to cocompact torsion-free Fuchsian groups. These maps are defined on the unit circle (the boundary of the Poincar\'e disk) by the generators of the group and have a finite set of discontinuities. We study the two forward orbits of each discontinuity point and show that for a family of such maps the \emph{cycle property} holds: the orbits 
coincide after finitely many steps.
We also show that for an open set of discontinuity points the associated two-dimensional natural extension maps possess global attractors with  \emph{finite rectangular structure}. These two properties belong to the list of ``good'' reduction algorithms, equivalence or implications between which were suggested by Don Zagier \cite{Z}.
\end{abstract}

\maketitle

\section{Introduction}
Let $\G$ be a finitely generated Fuchsian group of the first kind acting on the hyperbolic plane. We will use either the
upper half-plane model $\h$ or the unit disk model $\mathbb D$, and will denote the Euclidean boundary for either model by $\Sb$: for the upper half plane $\Sb=\partial(\h)=\PP^1(\R)$, and for the unit disk $\Sb=\partial(\mathbb D)=\mathbf{S}^1$. 

Let $\F$ be a fundamental domain for  $\G$ with an even number $N$ of sides  identified by the set of generators $G=\{T_1,\dots ,T_N\}$ of $\G$, and $\tau:\Sb\to G$ be a {surjective} map locally constant on $\Sb\setminus J$, where $J=\{x_1,\dots, x_N\}$ is an arbitrary set of jumps.  
A {\em boundary map} $f:\Sb\to\Sb$ is defined by $f(x)=\tau(x)x$. It is a piecewise fractional-linear map whose set of discontinuities  is $J$. 
Let $\Delta=\{(x,x)\mid x\in\Sb\}\subset \Sb\times\Sb$ be the diagonal of $\Sb\times\Sb$, and
$F:\Sb\times\Sb\setminus\Delta\to\Sb\times\Sb\setminus\Delta$ be given by
\[
F(x,y)=(\tau(y)x,\tau(y)y).
\]
This is a {\em (natural) extension} of $f$, and if
we identify $(x,y)\in \Sb\times\Sb\setminus\Delta$ with
an oriented geodesic from $x$ to $y$, we can think of
$F$ as a map on geodesics $(x,y)$ which 
 we will also call a {\em reduction map}.

Several years ago Don Zagier\cite{Z} proposed a list of possible notions of ``good" reduction algorithms associated to Fuchsian groups and conjectured equivalences or implications between them. 
In this paper we consider two of these notions, namely the properties that ``good" reduction algorithms should (i) satisfy the cycle property, and (ii) have an attractor with finite rectangular structure. 
We prove that {\em for each cocompact torsion-free Fuchsian group there exist families of reduction algorithms which satisfy these properties}. 
Thus our results are contributions towards Zagier's conjecture.

Although the statement that each Fuchsian group admits a ``good'' reduction algorithm  is 
not part of Zagier's conjecture, it is certainly related to it, and
for the purposes of this paper, we state it here.

\subsection*{Reduction Theory Conjecture for Fuchsian groups} For every Fuchsian group $\G$ there exist $\F, G$ as above, and an open set of $J's$ in $\Sb^N$ such that
 \begin{enumerate}
\item The map $F$ possesses a bijectivity domain $\Omega$
having a {\em finite rectangular structure}, i.e., bounded by non-decreasing step-functions with a finite 
number of steps.
\item Every point $(x,y)\in \Sb\times\Sb\setminus\Delta$  is mapped to $\Omega$ after finitely many iterations of $F$.
\end{enumerate}
\begin{rem} If property (2) holds, then $\Omega$ is a global attractor for the map $F$, i.e. 
\begin{equation}\label{att}
\Omega=\bigcap_{n=0}^\infty F^n(\Sb\times\Sb\setminus\Delta).
\end{equation}
\end{rem}
This conjecture was proved by the authors in \cite{KU3} for $\G=\sz$ and boundary maps associated to $(a,b)$-continued fractions. Notice that for some classical cases of continued fraction algorithms property (2) holds only for almost every point, while property (\ref{att}) remains valid.

 In this paper we address the conjecture for surface groups.  In the Poincar\'e unit disk model $\D$ endowed with the hyperbolic metric
 \begin{equation}\label{hypmetric}
 \frac{2|dz|}{1-|z|^2},
 \end{equation}
 let $\G$ be a Fuchsian group, i.e. a discrete group of orientation preserving isometries of $\D$, acting freely on $\D$ with $\G\backslash \D$ compact domain.
Such  $\G$ is called a {\em surface group}, and the quotient $\G\backslash \D$ is a compact surface of constant negative curvature $-1$ of a certain genus $g>1$.  
A classical (Ford) fundamental domain for $\G$ is a  $4g$-sided regular polygon centered at the origin (see {a sketch of} the construction in \cite{K2} in the manner of \cite{JS}, and for the complete proof see \cite{M}).
A more suitable for our purposes $(8g-4)$-sided fundamental domain $\mathcal F$ was described by Adler and Flatto in \cite{AF3}. They showed that all angles of $\mathcal F$ are equal to $\frac{\pi}2$ and, therefore, its sides are geodesic segments which satisfy the {\em extension condition} of Bowen and Series \cite{BoS}:  the geodesic extensions of these segments never intersect the interior of the tiling sets $\gamma\mathcal F$, $\gamma\in \G$. Figure \ref{fund} shows such a construction for $g=2$. 

\begin{figure}[htb]
\includegraphics[scale=1.1]{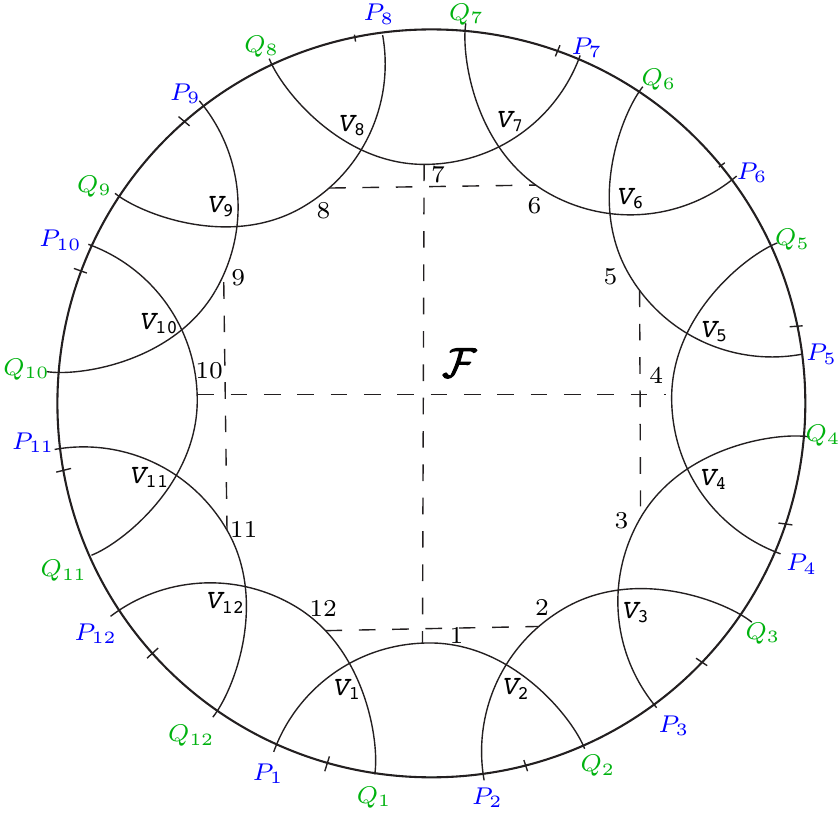}
\caption{The fundamental domain $\mathcal F$ for a genus $2$ surface}
 \label{fund}
\end{figure}

Using notations similar to \cite{AF3}, we label the sides of $\F$ in a counterclockwise order by numbers $1\leq i\leq 8g-4$, as they are arcs of the corresponding isometric circles of generators $T_i$. We denote the corresponding vertices of $\mathcal F$ by $V_i$, so that the side $i$ connects the vertices $V_i$ and $V_{i+1} \pmod {8g-4}$.
The identification of the sides is given by the pairing rule:
\[
\sigma(i)=\begin{cases}
4g-i \mod (8g-4) \text{ for odd } i \\

2-i \mod (8g-4) \text{ for even } i
\end{cases}\,.
\]

The generators $T_i$ associated to this fundamental domain {are M\"obius transformations satisfying} the following properties:
\begin{eqnarray}
& &T_{\sigma(i)}T_i=Id\label{r11}\\
& &T_i(V_i)=V_{\rho(i)}, \text{ where } \rho(i)=\sigma(i)+1\label{r12}\\
& &T_{\rho^3(i)}T_{\rho^2(i)}T_{\rho(i)}T_i=Id\label{r13}
\end{eqnarray}

We denote by $P_iQ_{i+1}$ the oriented (infinite) geodesic that extends the side $i$
to the boundary of the fundamental domain $\mathcal F$. {It is important to remark that 
${P_iQ_{i+1}}$ is the isometric circle for $T_i$, and $T_i\left({P_iQ_{i+1}}\right)={Q_{\sigma(i)+1}P_{\sigma(i)}}$ is the isometric circle for $T_{\sigma(i)}$ so that the inside of the former isometric circle is mapped to the outside of the latter.}

The counter-clockwise order of theses points on $\Sb$ is 
\begin{equation}\label{eq:BS}
P_1,Q_1, P_{2}, Q_2,\dots, P_{8g-4},Q_{8g-4},P_1.
\end{equation}

Bowen and Series \cite{BoS} defined the boundary map $f_{\bar P}:\Sb\ra \Sb$
\begin{equation}
f_{\bar P}(x)=T_i(x) \quad \text{if } P_i\le x<P_{i+1}\,.
\end{equation}
with the set of jumps $J=\bar{P}=\{P_1,\dots,P_{8g-4}\}$. They
showed that such a map is Markov with respect to the partition (\ref{eq:BS}),
expanding, and satisfies 
R\'enyi's distortion estimates, hence it admits a unique finite invariant ergodic measure equivalent to Lebesgue measure. 

Adler and Flatto \cite{AF3} proved the existence of an invariant domain  for the corresponding natural extension map $F_{\bar P}$, $\Omega_{\bar P}\subset \Sb\times \Sb$.
 Moreover, the set $\Omega_{\bar P}$ they identified has a regular geometric structure, what we call \emph{finite rectangular} (see Figure \ref{fig:BS}, with $\Omega_{\bar P}$ shown as a subset of $[-\pi,\pi]^2$). The maps $F_{\bar P}$ and $f_{\bar P}$ are ergodic\footnote{More precisely, $F_{\bar P}$  is a $K$-automorphism, property that is equivalent to $f_{\bar P}$ being an exact endo\-morphism.}. %
 Both Series \cite{S2} and Adler-Flatto \cite{AF3} explain how the boundary map can be used for coding symbolically the geodesic flow on $\mathbb D/\G$.

\begin{figure}[htb]
\includegraphics[scale=0.85 ]{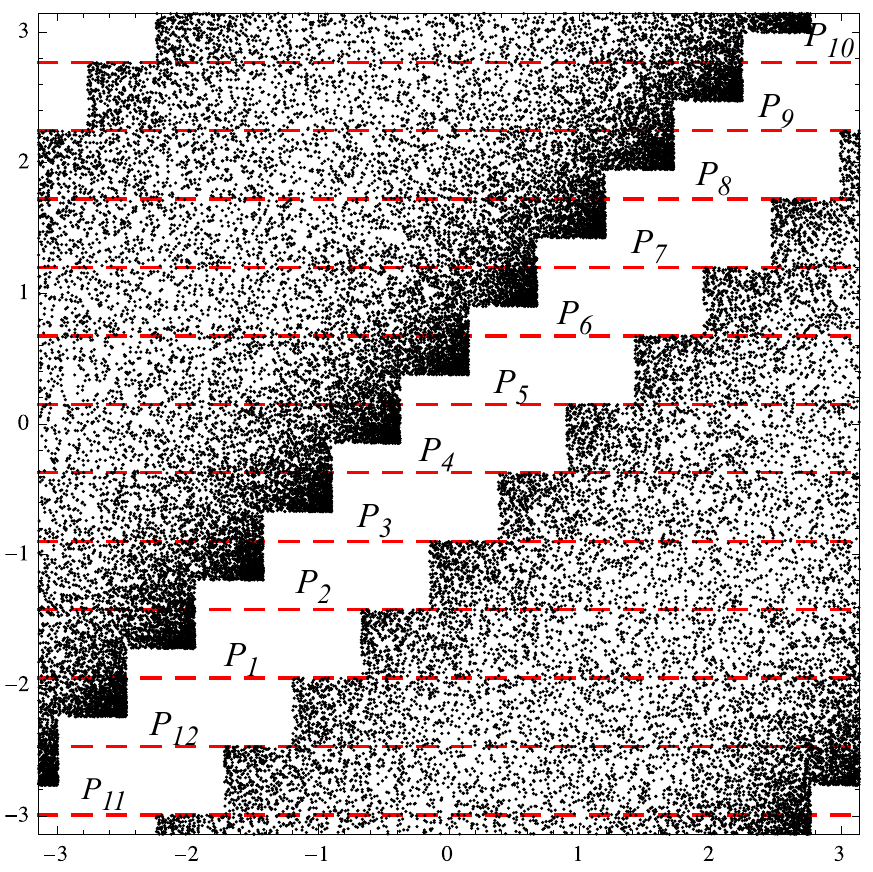}
\caption{Domain of the Bowen-Series map $F_{\bar P}$ as a subset of $[-\pi,\pi]^2$}
\label{fig:BS}
\end{figure}
\bigskip
\noindent{\bf Notations.} For $A,B\in \Sb$, the various intervals on $\Sb$ between $A$ and $B$ (with the counterclockwise order) will be denoted by $[A,B], (A,B], [A,B)$ and $(A,B)$. The geodesic (segment) from a point $C\in\Sb$ (or $\D$) to $D\in\Sb$ (or $\D)$ will be denoted by $CD$.

\medskip
Our object of study is a generalization of the Bowen-Series boundary map. We 
{consider an open set of jumps
\[
J=\bar A=\{A_1,\dots,A_{8g-4}\}
\]
with the only condition $A_i\in (P_i,Q_i)$,}
and define $f_{\bar A}:\Sb\ra \Sb$ by
\begin{equation} 
f_{\bar A}(x)=T_i(x)\quad  \text{if } A_i\le x<A_{i+1}\,,
\end{equation}
and the corresponding two-dimensional map:
\begin{equation}
F_{\bar A}(x,y)=(T_i(x),T_i(y)) \quad \text{if } A_i\le y<A_{i+1}\,.
\end{equation}

A key ingredient {in} analyzing map $F_{\bar A}$ is what we call the \textit{cycle property} of the partition points  $\{A_1,\dots,A_{8g-4}\}$. Such a property refers to the structure of the orbits of each $A_i$  that one can construct by tracking the two images $T_iA_i$ and $T_{i-1}A_i$ of these points of discontinuity of the map $f_{\bar A}$. It happens that some forward iterates of these two images  $T_iA_i$ and $T_{i-1}A_i$ under $f_{\bar A}$ coincide. This is another property from Zagier's list of ``good" reduction algorithms.

We state the cycle property result below and provide a proof in Section \ref{s:cycle}. 

\begin{thm}[Cycle Property]\label{cycle}
Each partition point $A_i\in (P_i,Q_i)$, $1\le i\le 8g-4$,  satisfies the cycle property, i.e., there exist positive integers $m_i, k_i$ such that
\[
f_{\bar A}^{m_i}(T_iA_i)=f_{\bar A}^{k_i}(T_{i-1}A_i).\]
\end{thm}

If a cycle closes up after one iteration
\begin{equation}\label{eq:sc}
f_{\bar A}(T_iA_i)=f_{\bar A}(T_{i-1}A_i),
\end{equation}
we say that the point $A_i$ satisfies the \emph{short cycle property}. 
Under this condition, we prove the following:
\begin{thm}[Main Result]\label{main}
If each partition point $A_i$ satisfies the short cycle property \eqref{eq:sc}, then
there exists a set $\Omega_{\bar A}\subset \Sb\times \Sb$ with the following properties:
\begin{enumerate}\item $\Omega_{\bar A}$ has a finite rectangular structure, and $F_{\bar A}$ is (essentially) bijective on $\Omega_{\bar A}$.
\item Almost every point $(x,y)\in \Sb\times \Sb\setminus \Delta$ is mapped to $\Omega_{\bar A}$ after finitely many iterations of $F_{\bar A}$, and $\Omega_{A}$ is a global attractor for the map $F_{\bar A}$, i.e., \[\Omega_{A}=\bigcap_{n=0}^\infty F_{\bar A}^n(\Sb\times \Sb\setminus \Delta).\]
\end{enumerate}
\end{thm}

\begin{figure}[htb]
\includegraphics[scale=1]{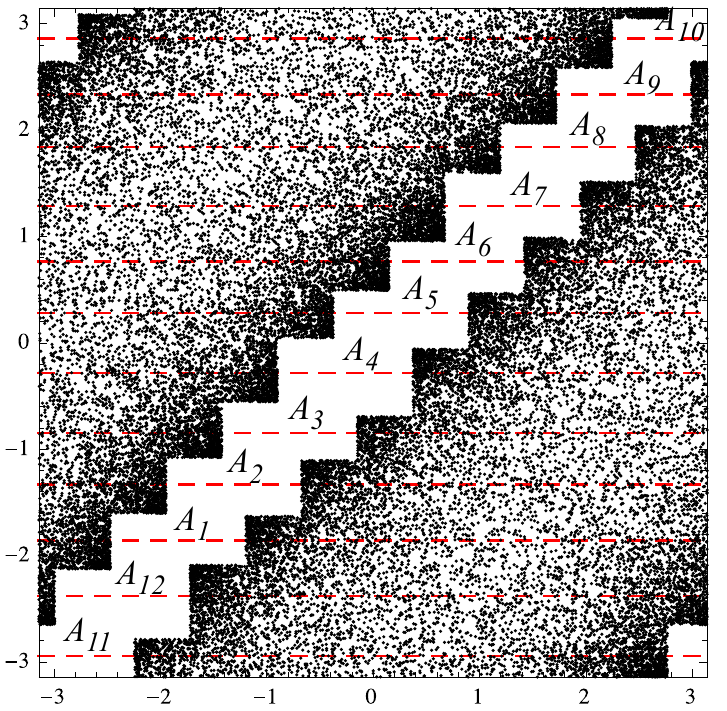}
\caption{Domain (and attractor) of the generalized Bowen-Series map $F_{\bar A}$}
\end{figure}

{Notice that the set of partitions satisfying the short cycle property contains an open set with this property, as explained in Remark \ref{aibi}. Thus we prove the Reduction Theory Conjecture. We believe that this result is true in greater generality, i.e., for all partitions $\bar A=\{A_i\}$ with $A_i\in (P_i,Q_i)$.}

\subsection*{Organization of the paper}

In Section \ref{s:BS} we prove properties (1) and (2) of the Reduction Theory Conjecture
for the classical Bowen-Series case when the partition points are given by the set $\bar P=\{P_i\}$. 
In Section \ref{s:cycle} we prove the cycle property for any partition $\bar A=\{A_i\}$ with $A_i\in (P_i,Q_i)$.
 In Section \ref{s:bijectivity} we determine the structure of the set $\Omega_{\bar A}$
in the case when the partition $\bar A$ satisfies the short cycle property and prove the bijectivity of the map $F_{\bar A}$ on $\Omega_{\bar A}$.  In Section \ref{s:trapping} we identify the trapping region for the map $F_{\bar A}$
and prove that every point in $\Sb\times \Sb\setminus \Delta$ is mapped to it after finitely many iterations of the map $F_{\bar A}$. And finally, in Section \ref{s:reduction} we prove that {almost} every point $\Sb\times \Sb\setminus \Delta$ is mapped to $\Omega_{\bar A}$ after finitely many iterations of the map $F_{\bar A}$ and complete the proof of Theorem~\ref{main}. In Section \ref{s:entropy} we apply our results to calculate the invariant probability measures for the maps $F_{\bar A}$ and $f_{\bar A}$.

\section{Bowen-Series case}\label{s:BS}
In this section we prove  properties (1) and (2) of  the Reduction Theory Conjecture  for the Bowen-Series classical case, where the partition $\bar A$ is given by  the set of points $\bar P=\{P_1,\dots,P_{8g-4}\}$. 

\begin{thm}\label{thm:BS}
The two-dimensional Bowen-Series map $F_{\bar P}$ satisfies properties {(1) and (2) of the Reduction Theory Conjecture.}
\end{thm}
Before we prove this theorem, we state a useful proposition that can be easily derived using the isometric circles and the conformal property of M\"obius transformations  (see also Theorem 3.4 of \cite{AF3}).
\begin{prop}\label{rem1} $T_i$ maps the points $P_{i-1}$, $P_i$, $Q_{i}$, $P_{i+1}$, $Q_{i+1}$, $Q_{i+2}$ respectively to $P_{\sigma(i)+1}$, $Q_{\sigma(i)+1}$, $Q_{\sigma(i)+2}$, $P_{\sigma(i)-1}$, $P_{\sigma(i)}$, $Q_{\sigma(i)}$. 
\end{prop}

\begin{proof}[Proof of Theorem \ref{thm:BS}]
In this case the set $\Omega_{\bar P}$ is determined by the  corner points located in each horizontal strip 
\[
\{(x,y)\in\Sb\times\Sb \mid y\in[P_i,P_{i+1})\}
\]
(see Figure \ref{fig:stripBS}) with coordinates  
$$(P_{i},Q_i) \text{ (upper part)}  \; \text{ and } \;  (Q_{i+2},P_{i+1}) \text{ (lower part)}.$$ 

\begin{figure}[htb]
\includegraphics[scale=1.1]{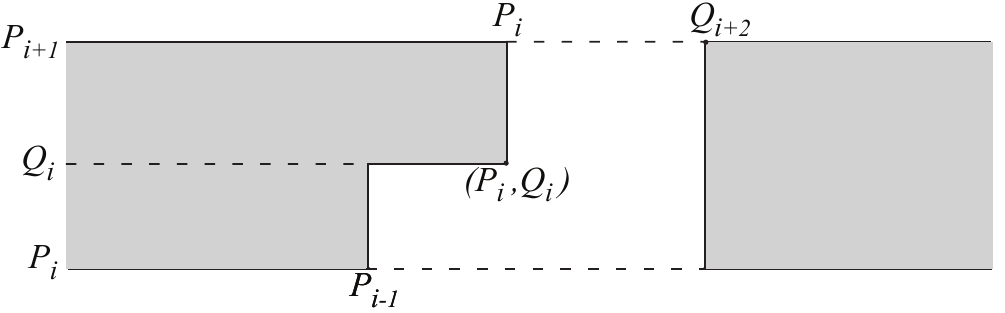}
\caption{Strip $y\in [P_i,P_{i+1}]$ of $\Omega_{\bar P}$}
\label{fig:stripBS}
\end{figure}

This set  {obviously has} a finite rectangular structure. One can also verify immediately the essential bijectivity, by investigating how different regions of $\Omega_{\bar P}$ are mapped by $F_{\bar P}$. More precisely we look at the strip $S_i$  of $\Omega_{\bar P}$ given by $y\in [P_i,P_{i+1}]$, and its image under $F_{\bar P}$, in this case $T_i$.

We consider the following decomposition of this strip: $\tilde S_i=[Q_{i+2},P_{i-1}]\times [P_i,Q_{i}]$ (red rectangular horizontal piece),
$\hat S_i=[Q_{i+2},P_i]\times [Q_i,P_{i+1}]$ (green horizontal rectangular piece). Now
\begin{align*}
T_i(\tilde S_i)& = [T_iQ_{i+2},T_iP_{i-1}]\times [T_iP_i,T_iQ_i]=[Q_{\sigma(i)},P_{\sigma(i)+1}]\times [Q_{\sigma(i)+1},Q_{\sigma(i)+2}]\\
T_i(\hat S_i) & = [T_iQ_{i+2},T_iP_{i}]\times [T_iQ_i,T_iP_{i+1}]=[Q_{\sigma(i)},Q_{\sigma(i)+1}]\times [Q_{\sigma(i)+2},P_{\sigma(i)-1}]
\end{align*}
Therefore $T_i(S_i)$ is a complete vertical strip in $\Omega_{\bar P}$, with $Q_{\sigma(i)}\le x\le Q_{\sigma(i)+1}$. {This completes the proof of the property (1).}

\begin{figure}[htb]
\includegraphics[scale=0.9]{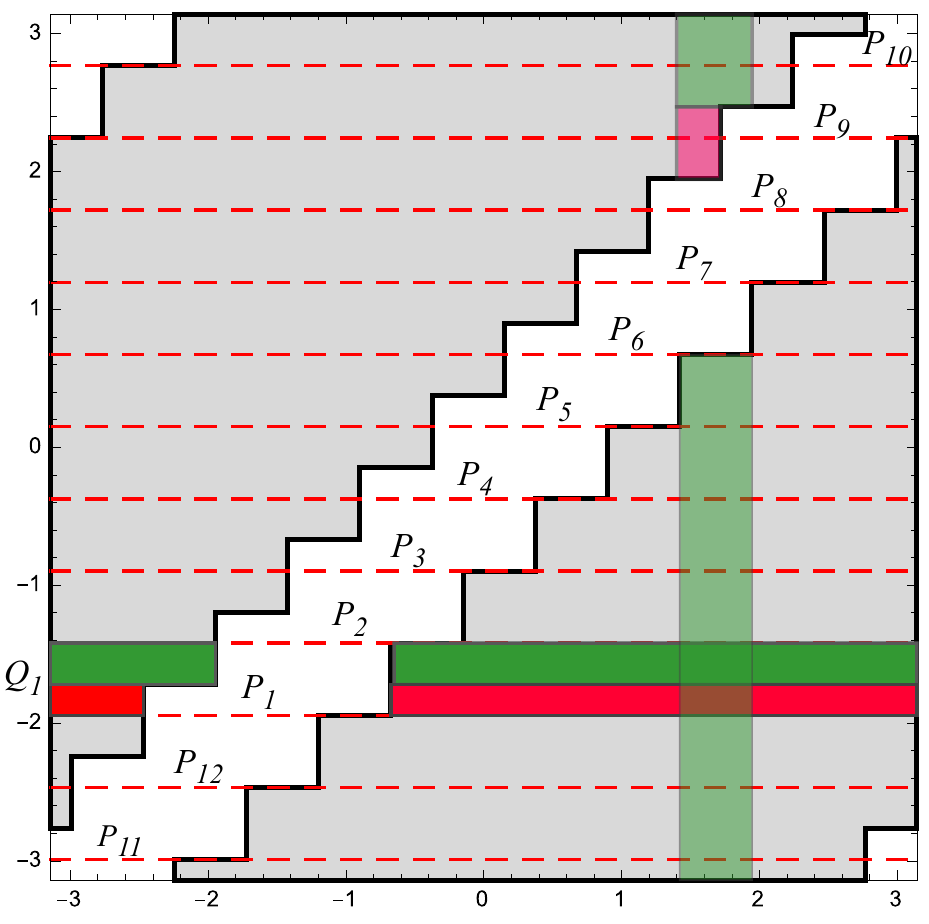}
\caption{Bijectivity of the Bowen-Series map $F_{\bar P}$}
\label{fig:BSbij}
\end{figure}

\smallskip

We now prove property (2) for the set $\Omega_P$. 

Consider $(x,y)\in \Sb\times \Sb\setminus \Delta$. Notice that there exists $n(x,y)>0$ such that the two values $x_n,y_n$ obtained from the $n$th iterate of $F_{\bar P}$,  $(x_n,y_n)=F^n_{\bar P}(x,y)$, are not inside the same isometric circle; in other words, $(x_n,y_n)\not\in X_i=[P_{i},Q_{i+1}]\times [P_{i},P_{i+1})$ for all  $1\le i\le 8g-4$. Indeed, if one assumes that both coordinates $(x_n,y_n)=F^n_{\bar P}(x,y)$ belong to such a  set $X_i$ for all $n\ge 0$, each time we iterate the pair $(x_n,y_n)$ we apply one of the maps $T_{i}$ which is expanding in the interior of  its isometric circle.  Thus the distance between $x_n$ and $y_n$ would grow sufficiently for the points to be inside different isometric circles. Therefore, there exists $n>0$ such that $y_n$ is in some
interval $[P_{i},P_{i+1}) \subset [P_{i},Q_{i+1}]$ and $x_n\not \in [P_{i},Q_{i+1}]$.

Notice that, from the definition of $\Omega_{\bar P}$, in order to prove the attracting property, we need to analyze the situations $(x_n,y_n)\in [P_{i-1},P_{i}]\times [P_{i},Q_i]$ and $(x_n,y_n)\in [Q_{i+1},Q_{i+2}]\times[P_i,P_{i+1})$ and show that a forward iterate lands in $\Omega_{\bar P}$.

\medskip

\noindent \textbf{Case I.}  If  $(x_n,y_n)\in [Q_{i+1},Q_{i+2}]\times[P_i,P_{i+1})$, then
$$F_{\bar P}(x_n,y_n)\in [T_iQ_{i+1},T_iQ_{i+2}]\times [T_iP_i,T_iP_{i+1})=[P_{\sigma(i)},Q_{\sigma(i)}]\times[Q_{\sigma (i)+1},P_{\sigma(i)-1}).$$
The subset $[P_{\sigma(i)},Q_{\sigma(i)}]\times [Q_{\sigma(i)+1},P_{\sigma(i)-2}]$ is included in $\Omega_{\bar P}$ so we only need to analyze the situation $(x_{n+1},y_{n+1})\in [P_{k+2},Q_{k+2}]\times [P_k,P_{k+1})$, where $k=\sigma(i)-2$.
 Then
$$(x_{n+2},y_{n+2})=F^2_P(x_n,y_n)=T_kT_i(x_n,y_n)\in [T_kP_{k+2},Q_{\sigma(k)}]\times[Q_{\sigma(k)+1},P_{\sigma(k)-1})\,.$$
Notice that $T_kP_{k+2}\in [P_{\sigma(k)},Q_{\sigma(k)}]$. The subset $[T_kP_{k+2},Q_{\sigma(k)}]\times[Q_{\sigma(k)+1},P_{\sigma(k)-2}]$  is included in $\Omega_{\bar P}$ so we only need to analyze the situation \[(x_{n+2},y_{n+2})\in [T_kP_{k+2},Q_{\sigma(k)}]\times [P_{\sigma(k)-2},P_{\sigma(k)-1})\subset [P_{\sigma(k)},Q_{\sigma(k)}]\times[P_{\sigma(k)-2},P_{\sigma(k)-1}).\]

Notice that $\sigma(k)-2=\sigma(\sigma(i)-2)-2=i$ (direct verification), so we are back to analyzing the situation $(x_{n+2},y_{n+2})\in [P_{i+2},Q_{i+2}]\times[P_i,P_{i+1})$. The boundary map $f_{\bar P}$ is expanding, so it is not possible for the images of the interval $(y_n,P_{i+1})$ (on the $y$-axis) to  alternate indefinitely between the intervals $[P_i,P_{i+1}]$ and $[P_{\sigma(i)-2},P_{\sigma(i)-1}]$, where $T_iP_{i+1}=P_{\sigma(i)-1}$. 

This means that either some even iterate  \[F^{2m}(x_n,y_n)\in [P_{i+2},Q_{i+2}]\times[Q_{i+3},P_{i})\subset \Omega_{\bar P}\] or some odd iterate \[F^{2m+1}(x_n,y_n)\in  [P_{\sigma(i)},Q_{\sigma(i)}]\times[Q_{\sigma(i)+1},P_{\sigma(i)-2}]\subset \Omega_{\bar P}.\]

\medskip

\noindent \textbf{Case II.}  If $(x_n,y_n)\in [P_{i-1},P_{i}]\times [P_{i},Q_i]$, then
\[F_{\bar P}(x_n,y_n)\in [T_iP_{i-1},T_iP_i]\times [T_iP_i,T_iQ_i]=[P_{\sigma(i)+1},Q_{\sigma(i)+1}]\times[Q_{\sigma (i)+1},Q_{\sigma(i)+2}]\,.\]
There are two subcases that we need to analyze: \[(a) \;\;(x_{n+1},y_{n+1})\in [P_k,Q_k]\times[Q_k,P_{k+1}) \quad (b) \;\; (x_{n+1},y_{n+1})\in [P_k,Q_k]\times[P_{k+1},Q_{k+1}],\] where $k=\sigma(i)+1$.

\medskip

\noindent \textbf{Case (a)} If $(x_{n+1},y_{n+1})\in [P_k,Q_k]\times[Q_k,P_{k+1}]$, then
$$(x_{n+2},y_{n+2})\in T_k\left(  [P_k,Q_k]\times[Q_k,P_{k+1})\right)=[Q_{\sigma(k)+1},Q_{\sigma(k)+2}]\times[Q_{\sigma (k)+2},P_{\sigma(k)-1}]\,.$$
Notice that $\sigma(k)+1=\sigma(\sigma(i)+1)+1=4g+i-2$ (direct verification), so when analyzing the situation $(x_{n+2},y_{n+2})\in [Q_{4g+i-2},Q_{4g+i-1}]\times[Q_{4g+i-1},P_{4g+i-4})$  the only problematic region is $(x_{n+2},y_{n+2})\in [P_{4g+i-1},Q_{4g+i-1}]\times [Q_{4g+i-1},Q_{4g+i}]$.

\smallskip

\noindent \textbf{Case (b)} If  $(x_{n+1},y_{n+1})\in [P_k,Q_k]\times[P_{k+1},Q_{k+1}]$, then
\[
\begin{aligned}
(x_{n+2},y_{n+2})&\in T_{k+1}\left ( [P_k,Q_k]\times[P_{k+1},Q_{k+1}]\right)\\
&=[P_{\sigma(k+1)+1},T_{k+1}Q_{k}]\times[Q_{\sigma (k+1)+1},Q_{\sigma(k+1)+2}].
\end{aligned}
\]
Notice that $T_{k+1}Q_k\in [P_{\sigma(k+1)+1},Q_{\sigma(k+1)+1}]$  and $\sigma(k+1)+1=i-1$ (direct verification) so we are left to investigate $(x_{n+2},y_{n+2})\in [P_{i-1},Q_{i-1}]\times [Q_{i-1},Q_{i}]$.

\medskip

To summarize, we started with $(x_{n+1},y_{n+1})\in [P_{\sigma(i)+1},Q_{\sigma(i)+1}]\times[Q_{\sigma (i)+1},Q_{\sigma(i)+2}]\,$ and found two situations that need to be analyzed: $(x_{n+2},y_{n+2})\in [P_{i-1},Q_{i-1}]\times [Q_{i-1},Q_{i}]$ and  $(x_{n+2},y_{n+2})\in [P_{4g+i-1},Q_{4g+i-1}]\times [Q_{4g+i-1},Q_{4g+i}]$.

We prove in what follows that it is not possible for all future iterates $F^m(x_n,y_n)$  to belong to the sets of type $[P_k,Q_k]\times[Q_{k},Q_{k+1}]$. First, it is not possible for all $F^m(x_n,y_n)$ (starting with some $m>0$) to belong only to type-a sets 
$[P_{k_m},Q_{k_m}]\times[Q_{k_m},P_{k_m+1}]$, where the sequence $\{k_m\}$ is defined recursively as $k_{m}=\sigma(k_{m-1})+2$, because such a set is included in the isometric circle $X_{k_m}$, and the argument at the beginning of the proof disallows such a situation.

Also, it is not possible for all $F^m(x_n,y_n)$ (starting with some $m>0$) to belong only to type-b sets $[P_{k_m},Q_{k_m}]\times[P_{k_m+1},Q_{k_m+1}]$, where $k_{m}=\sigma(k_{m-1}+1)+1$: this would imply that the pairs of points $(y_{n+m},Q_{k_{n+m}+1})$ (on the $y$-axis) will belong to the same interval $[P_{k_{n+m}+1},Q_{k_{n+m}+1}]$ which is impossible due to expansiveness property of the map $f_{\bar P}$. Therefore, there exists a pair  $(x_{l},y_{l})$ in the orbit of $F^m(x_n,y_n)$ such that 
$$(x_{l},y_{l})\in  [P_{j},Q_{j}]\times[P_{j+1},Q_{j+1}] \text{ (type-b)}$$ for some $1\le j\le 8g-4$ and $$(x_{l+1},y_{l+1})\in [P_{j'},T_{j+1}Q_{j}]\times[Q_{j'},P_{j'+1}] \subset [P_{j'},Q_{j'}]\times[Q_{j'},P_{j'+1}] \text{ (type-a)},$$ where $j'=\sigma(j+1)+1$. Then 
$$(x_{l+2},y_{l+2})\in T_{j'}([P_{j'},T_{j+1}Q_{j}] \times[Q_{j'},P_{j'+1}])=[Q_{j''},T_{j'}T_{j+1}Q_{j}] \times[Q_{j''+1},P_{j''-2}]$$
where $j''=\sigma(j')+1$. 

Using the results of the Appendix (Corollary \ref{lem:app}), we have that the arc length distance  
$$\ell(P_{j'},T_{j+1}Q_j)=\ell(T_{j+1}P_j,T_{j+1}Q_j)<\frac{1}{2}\ell(P_{j'},Q_{j'}).$$
Now we can use Corollary \ref{new} (ii) applied to the point $T_{j+1}Q_j\in[P_{j'},Q_{j'}]$ to conclude that  $T_{j'}T_{j+1}Q_j\in [Q_{j''},P_{j''+1}]$. 
Therefore $(x_{l+2},y_{l+2})\in \Omega_{\bar P}$. {This completes the proof of the property (2).} 
\end{proof}

\begin{rem} {One can prove along the same lines that if the partition $\bar A$ is given by the set $\bar Q=\{Q_1,\dots, Q_{8g-4}\}$, the properties (1) and (2) of the Reduction Theory Conjecture also hold.}
\end{rem}

\section{The cycle property}\label{s:cycle}

The map $f_{\bar A}$ is discontinuous at $x=A_i$, $1\le i \le 8g-4$. We associate to each point $A_i$  two forward orbits:  the {\em upper orbit} $\OO_u(A_i)=\{f_{\bar A}^{n}(T_iA_i)\}_{n\ge 0}$, and the {\em lower orbit}  $\OO_\ell(A_i)=\{f_{\bar A}^n(T_{i-1}A_i)\}_{n\ge 0}$. We use the convention that if an orbit hits one of the discontinuity points $A_j$, then the next iterate is computed according to the left or right location: for example, if the lower orbit of $A_i$ hits some $A_j$, then the next iterate will be $T_{j-1}A_j$, and if the upper orbit of $A_i$ hits some $A_j$ then the next iterate is $T_jA_j$.

Now we explore  the patterns in the above orbits. The following property plays an essential role in studying the maps $f_{\bar A}$ and $F_{\bar A}$.
\begin{defn}\label{def:cycles}
We say that the point $A_i$ has  the {\em cycle property} if for some non-negative integers $m_i,k_i$
\[
f_{\bar A}^{m_i}(T_iA_i)=f_{\bar A}^{k_i}(T_{i-1}A_i)=:c_{A_i}.
\]
We will refer to 
the set
\[
\{T_iA_i, f_{\bar A}T_iA_i,\dots ,f_{\bar A}^{m_i-1}T_iA_i\}
\]
as the {\em upper  side of the $A_i$-cycle}, the set
\[
\{T_{i-1}A_i, f_{\bar A}T_{i-1}A_i,\dots ,f_{\bar A}^{k_i-1}T_{i-1}A_i\}
\]
as the {\em lower side of the $A_i$-cycle}, and to $c_{A_i}$ as the {\em end of the $A_i$-cycle}.

\end{defn}

The main goal of this section is to prove Theorem \ref{cycle} (cycle property) stated in the Introduction. First, we prove some preliminary results.
\begin{lem}
The following identity holds
\begin{equation}\label{shortcycle}
T_{\sigma(i)+1}T_i=T_{\sigma(i-1)-1}T_{i-1}
\end{equation}
\end{lem}
\begin{proof}
Using relation \eqref{r13} stated in the Introduction, we have that $$T_{\rho(i)}T_i=T^{-1}_{\rho^2(i)}T^{-1}_{\rho^3(i)}$$ (where $\rho(i)=\sigma(i)+1$),  so it is enough to show that 
$T^{-1}_{\rho^2(i)}=T_{\sigma_{i-1}-1}$ and $T^{-1}_{\rho^3(i)}=T_{i-1}$. For that we analyze the two parity cases.

\medskip

\textbf{If $i$ is odd}, we have the following identities $\!\!\mod (8g-4)$:
\begin{align*}
\rho(i)&=\sigma(i)+1=4g-i+1 \text{ (even)}\\
\rho^2(i)&=\sigma(4g-i+1)+1=2-(4g-i+1)+1=2-4g+i=4g-2+i \text{ (odd)}\\
\rho^3(i)&=\sigma(2-4g+i)+1=4g-(2-4g+i)+1=8g-1-i=3-i \text{ (even)}
\end{align*}
Since $\sigma(i-1)=3-i=\rho^3(i)$, one has $T^{-1}_{\rho^3(i)}=T_{i-1}$ by using \eqref{r12}. Also, $\sigma(i-1)-1=2-(i-1)-1=2-i$ and $\sigma(\rho^2(i))=2-i$, hence $T^{-1}_{\rho^2(i)}=T_{\sigma(i-1)-1}$.

\medskip

\textbf{If $i$ is even}, we have the following identities $\!\!\mod (8g-4)$:
\begin{align*}
\rho(i)&=\sigma(i)+1=3-i \text{ (odd)}\\
\rho^2(i)&=\sigma(3-i)+1=4g-(3-i)+1=4g-2+i \text{ (even)}\\
\rho^3(i)&=\sigma(4g-2+i)+1=2-(4g-2+i)+1=5-4g-i=4g+1-i \text{ (odd)}
\end{align*}
Since $\sigma(i-1)=4g-(i-1)=\rho^3(i)$, one has $T^{-1}_{\rho^3(i)}=T_{i-1}$ by using \eqref{r12}. Also, 
$\sigma(i-1)-1=4g-i$ and $\sigma(\rho^2(i))=4g-i$, hence $T^{-1}_{\rho^2(i)}=T_{\sigma(i-1)-1}$.

Identity \eqref{shortcycle} has been proved for both cases.
\end{proof}

\begin{rem}
By introducing the notation $\theta(i)=\sigma(i)-1$, relation \eqref{shortcycle} can be written
\begin{equation}\label{shortcycle2}
T_{\rho(i)}T_i=T_{\theta(i-1)}T_{i-1},
\end{equation}
which will simplify further calculations. 
\end{rem}

\begin{lem}\label{lem:rel1}
For any $1\le i\le 8g-4$, $\theta(\theta(i-1)-1))=i$ and $\rho(\rho(i)+1)+1=i$.
\end{lem}
\begin{proof}
Immediate verification.
\end{proof}
\begin{lem}\label{lem:period2}
The relations $f^2_{\bar A}(P_i)=P_i$ and $f^2_{\bar A}(Q_i)=Q_i$ hold for all $i$. In addition, $f_{\bar A}(P_i)=P_i$ if $i\in \{1, 2g, 4g-1, 6g-2\}$, and $f_{\bar A}(Q_i)=Q_i$ if $i\in \{2, 2g+1, 4g, 6g-1\}$.
\end{lem}
\begin{proof}
We have
\[
f^2_{\bar A}(P_i)=f_{\bar A}(T_{i-1}P_i)=f_{\bar A}(P_{\theta(i-1)})=P_{\theta(\theta(i-1)-1)}=P_{i}
\]
and 
\[
f^2_{\bar A}(Q_i)=f_{\bar A}(T_{i}Q_i)=f_{\bar A}(Q_{\rho(i)+1})=Q_{\rho(\rho(i)+1)+1}=Q_i
\]
by Lemma \ref{lem:rel1}. The second part follows easily, too.
\end{proof}
\begin{proof}[Proof of Theorem \ref{cycle}] Let us analyze the upper and lower orbits of  $A_i$.
By Proposition \ref{rem1} and the orientation preserving property of the M\"obius transformations, we have
\begin{equation}\label{eq:images}
T_i[P_i, Q_i]=[Q_{\rho(i)},Q_{\rho(i)+1}], \quad T_{i-1}[P_i, Q_i]=[P_{\theta(i-1)},P_{\theta(i-1)+1}],
\end{equation}
therefore 
\begin{equation}\label{eq:sets}
T_iA_i\in \left(Q_{\rho(i)}, Q_{\rho(i)+1}\right)\;,\; T_{i-1}A_i\in \left(P_{\theta(i-1)}, P_{\theta(i-1)+1}\right)
\end{equation}
Depending on whether $T_iA_i\in  (Q_{\rho(i)}, A_{\rho(i)+1})$ or $T_iA_i\in  {[}A_{\rho(i)+1}, Q_{\rho(i)+1})$ 
we have either
\[
f_{\bar A}(T_iA_i)=T_{\rho(i)}T_iA_i \text{ or } f_{\bar A}(T_iA_i)=T_{\rho(i)+1}T_iA_i\,.
\]
Also, depending on whether $T_{i-1}A_i\in  (P_{\theta(i-1)}, A_{\theta(i-1)}{]}$ or $T_{i-1}A_i\in  (A_{\theta(i-1)}, P_{\theta(i-1)+1})$ we have either
\[
f_{\bar A}(T_{i-1}A_i)=T_{\theta(i-1)-1}T_{i-1}A_i \text{ or } f_{\bar A}(T_{i-1}A_i)=T_{\theta(i-1)}T_iA_i\,.
\]

Notice that in the case when $T_iA_i\in  (Q_{\rho(i)}, A_{\rho(i)+1})$ and $T_{i-1}A_i\in  (A_{\theta(i-1)}, P_{\theta(i-1)+1})$ the cycle property holds immediately with $m_i=k_i=1$, by using relation \eqref{shortcycle2}. 

We are left to analyze the cases $T_iA_i\in [A_{\rho(i)+1}, Q_{\rho(i)+1})$ or $T_{i-1}A_i\in  (P_{\theta(i-1)}, A_{\theta(i-1)}]$. 

\begin{figure}[htb]
\includegraphics{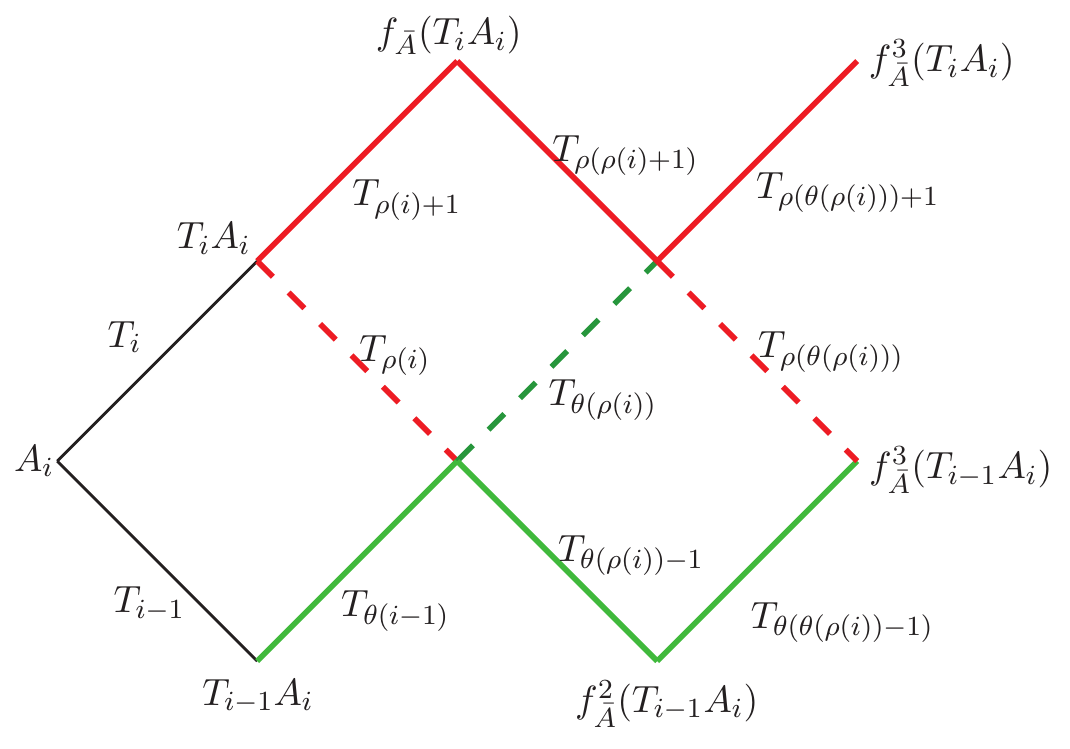}
\caption{The first iterates of the upper and lower orbits of $A_i$}
\end{figure}

\begin{lem}\label{lem:imp}
Given $x\in (P_i,Q_i)$ then one cannot have $T_ix\in [A_{\rho(i)+1}, Q_{\rho(i)+1})$ \underline{and} $T_{i-1}x\in  (P_{\theta(i-1)}, A_{\theta(i-1)}]$ simultaneously.
\end{lem}
\begin{proof} Let $M_i$ be the midpoint of $(P_i,Q_i)$. By Corollary \ref{new} of the Appendix, there exists $a_i\in(M_i,Q_i)$ such that $T_i(a_i)=P_{\rho(i)+1}$ and $b_i\in(P_i,M_i)$ such that $T_{j-1}(b_j)=Q_{\theta(j-1)}$. 

Since $A_{\rho(i)+1}\in (P_{\rho(i)+1}, Q_{\rho(i)+1})$ and $A_{\theta(i-1)}\in (P_{\theta(i-1)},Q_{\theta(i-1)}$, in order for
$T_i x\in [A_{\rho(i)+1}, Q_{\rho(i)+1})$, $x$ must be in $(a_i,Q_i)$, and   in order for $T_{i-1}x\in  (P_{\theta(i-1)}, A_{\theta(i-1)}]$, $x$ must be in $(P_i,b_i)$. The lemma follows from the fact that these intervals are disjoint.
\end{proof}

\begin{lem}\label{lem:imp2}$\;$
\begin{itemize}
\item[(i)] Assume $x\in [A_j,Q_j)$ and $T_{j-1}(x)\in (P_{\theta(j-1)},A_{\theta(j-1)}]$, then 
$$T_{\theta(j-1)-1}T_{j-1}(x)\in (x,P_{j+1}).$$
\item[(ii)] Assume $x\in (P_j,A_j]$ and $T_{j}(x)\in [A_{\rho(j)+1},Q_{\rho(j)+1})$, then 
$$T_{\rho(j)+1}T_{j}(x)\in (Q_{j-1},x).$$
\end{itemize}
\end{lem}
\begin{proof}
(i) Notice that $T_{\theta(j-1)-1}T_{j-1}(P_j)=f^2_{\bar A}(P_j)=P_j$ by Lemma \ref{lem:period2}. Also  
$T_{j-1}(x)\in (P_{\theta(j-1)},Q_{\theta(j-1)})$ therefore
$$T_{\theta(j-1)-1}T_{j-1}(x)\in  T_{\theta(j-1)-1}(P_{\theta(j-1)},Q_{\theta(j-1)})=(P_{j},P_{j+1})$$ by \eqref{eq:sets} and the fact that $\theta(\theta(j-1)-1)=j$ by Lemma \ref{lem:rel1}. It follows that
$$(T_{\theta(j-1)-1}T_{j-1})[P_j,x]=[P_j, T_{\theta(j-1)-1}T_{j-1}(x)]\subset [P_{j},P_{j+1}]\,.$$ 
Since $T_{\theta(j-1)-1}T_{j-1}$ expands $[P_j,x]$ we get
$T_{\theta(j-1)-1}T_{j-1}(x)\in (x,P_{j+1})$. 

Part (ii) can be proved similarly.  
\end{proof}

We continue the proof of the theorem and assume the situation \[T_iA_i\in [A_{\rho(i)+1}, Q_{\rho(i)+1}).\] Lemma \ref{lem:imp} implies that $T_{i-1}A_i{\notin}  (P_{\theta(i-1)}, A_{\theta(i-1)}]$, {i.e. $T_{i-1}A_i\in(A_{\theta(i-1)},P_{\theta(i-1)+1})$.}
 Notice that $f_{\bar A}(T_{i-1}A_i)$ can be rewritten as $T_{\rho(i)}T_iA_i$ by Lemma \ref{shortcycle}, and the beginning of the two orbits of $A_i$ are given by
\[
\OO_u(A_i)=\{T_iA_i, T_{\rho(i)+1}T_iA_i,\dots\},\quad
\OO_l(A_i)=\{T_{i-1}A_i, T_{\rho(i)}T_iA_i,\dots\}\,.
\]

We can now apply Lemma \ref{lem:imp2} part (ii) for $x=A_i$ to obtain that 
$$f_{\bar A}(T_iA_i)=T_{\rho(i)+1}T_iA_i\in (Q_{i-1},A_i),$$ therefore
$f^2_{\bar A}(T_iA_i)=T_{\rho(\rho(i)+1)}T_{\rho(i)+1}(T_iA_i)$ (recalling that $\rho(\rho(i)+1)=i-1$).

On the other hand $T_{\rho(i)}T_iA_i\in \left(P_{\theta(\rho(i))}, P_{\theta(\rho(i))+1}\right)$. Depending on whether 
$T_{\rho(i)}T_iA_i\in  \left(P_{\theta(\rho(i))}, A_{\theta(\rho(i))}\right]$ or $T_{\rho(i)}T_iA_i\in  \left(A_{\theta(\rho(i))}, P_{\theta(\rho(i))+1}\right)$ we have that 
$$f_{\bar A}(T_{\rho(i)}T_iA_i)=T_{\theta(\rho(i))-1}(T_{\rho(i)}T_iA_i) \text{ or } f_{\bar A}(T_{\rho(i)}T_iA_i)=T_{\theta(\rho(i))}(T_{\rho(i)}T_iA_i)\,.$$

In the {latter} case, the cycle property holds, by using relation \eqref{shortcycle2}: we have $f^2_{\bar A}(T_iA_i)=f^2_{\bar A}(T_{i-1}A_i)$, i.e. 
$$
T_{\rho(\rho(i)+1)}T_{\rho(i)+1}(T_iA_i)=T_{\theta(\rho(i))}T_{\rho(i)}(T_iA_i)\,.
$$
We have
\[
\begin{split}
\OO_u(A_i)&=\{T_iA_i, T_{\rho(i)+1}T_iA_i,T_{\theta(\rho(i))}(T_{\rho(i)}T_iA_i)\dots\}\\
\OO_l(A_i)&=\{T_{i-1}A_i, T_{\rho(i)}T_iA_i,T_{\theta(\rho(i)-1{)}}(T_{\rho(i)}T_iA_i)\dots\}\,.
\end{split}
\]
\begin{prop}
Assume that  $T_iA_i\in [A_{\rho(i)+1}, Q_{\rho(i)+1})$, and $A_i$ does not satisfy the cycle property up to iteration $2M+2$. Let $\psi_n=(\theta\circ\rho)^n$. 
 Then, for any $0\le n\le M$,
\begin{equation}\label{eq:odd}
\begin{split}
&f_{\bar A}^{2n}(T_iA_i)\in [A_{\rho(\psi_n(i))+1},Q_{\rho(\psi_n(i))+1})\\
&f_{\bar A}^{2n+1}(T_iA_i)=T_{\rho(\psi_n(i))+1}(f_{\bar A}^{2n}(T_iA_i))\\
&f_{\bar A}^{2n+1}(T_{i-1}A_i)=T_{\theta(\psi_{n}(i)-1)}(f_{\bar A}^{2n}(T_{i-1}A_{i}))=T_{\rho(\psi_n(i))}(f_{\bar A}^{2n}(T_iA_i))
\end{split}
\end{equation}
\begin{equation}\label{eq:even}
\begin{split}
&f_{\bar A}^{2n+1}(T_{i-1}A_i)\in (P_{\psi_{n+1}(i)},A_{\psi_{n+1}(i)}]\\
&f_{\bar A}^{2n+2}(T_iA_i)=T_{\rho(\psi_{n}(i))+1}(f_{\bar A}^{2n+1}(T_iA_i))=T_{\psi_{n+1}(i)}(f_{\bar A}^{2n+1}(T_{i-1}A_i))\\
&f_{\bar A}^{2n+2}(T_{i-1}A_i)=T_{\psi_{n+1}(i)-1}(f_{\bar A}^{2n+1}(T_{i-1}A_i)) 
\end{split}
\end{equation} 
\end{prop}
\begin{figure}[htb]
\includegraphics{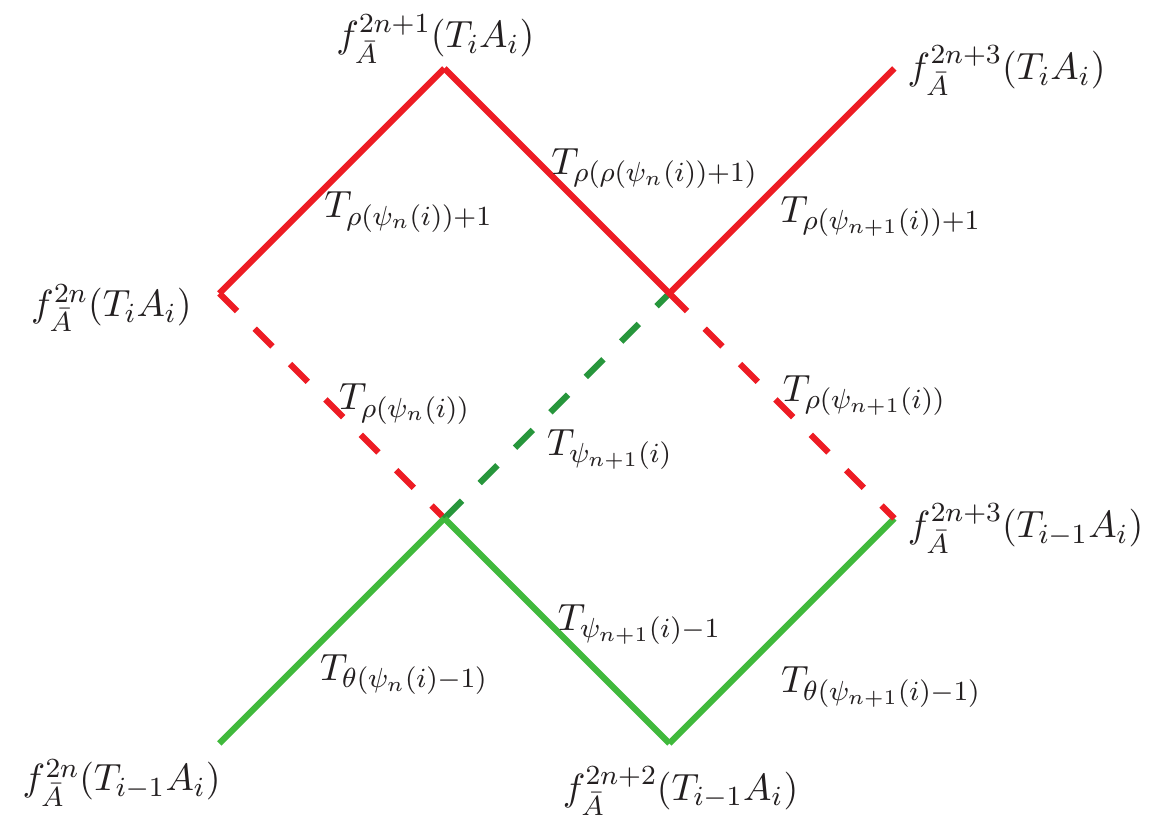}
\caption{Iterates of upper and lower orbits of $A_i$}
\end{figure}
\begin{proof}
We prove this by induction. The case $n=0$ has been already presented above ($\psi_0(i)=i$).
Assume now that the relations are true for $k=1,2,\dots, n<M$. We analyze the case $k=n+1$.
Let $\ell= \psi_n(i)$. First, notice that
\[
\begin{split}
f_{\bar A}^{2n+2}(T_{i-1}A_i)&=T_{\psi_{n+1}(i)-1}(f_{\bar A}^{2n+1}(T_{i-1}A_i))=T_{\psi_{n+1}(i)-1}T_{\rho(\psi_n(i))}(f_{\bar A}^{2n}(T_iA_i))\\
&=T_{\theta(\rho(\ell))-1}T_{\rho(\ell)}(f_{\bar A}^{2n}(T_iA_i))
\end{split}
\]
Since  $$f_{\bar A}^{2n}(T_iA_i)\in [A_{\rho(\psi_n(i))+1},Q_{\rho(\psi_n(i))+1})=[A_{\rho(\ell)+1},Q_{\rho(\ell)+1})$$ and $$T_{\rho(\ell)}(f_{\bar A}^{2n}(T_iA_i))=f_{\bar A}^{2n+1}(T_{i-1}A_i)\in  (P_{\theta(\rho(l))},A_{\theta(\rho(l))}]$$ we can apply Lemma \ref{lem:imp2} part (i) for $x=f_{\bar A}^{2n}(T_iA_i)$, $j=\rho(\ell)+1$ to conclude that
$f_{\bar A}^{2n+2}(T_{i-1}A_i)\in (A_{\rho(\ell)+1},P_{\rho(\ell)+2})$ and
\begin{equation}\label{eq:2n3}  
f_{\bar A}^{2n+3}(T_{i-1}A_i)=T_{\rho(\ell)+1}(f_{\bar A}^{2n+2}(T_{i-1}A_i))=T_{\theta(\psi_{n+1}(i)-1)}(f_{\bar A}^{2n+2}(T_{i-1}A_{i}))
\end{equation}
because $\rho(\ell)+1=\theta(\theta(\rho(\ell))-1)=\theta(\psi_{n+1}(i)-1)$. 

Since $$f_{\bar A}^{2n+2}(T_iA_i)=T_{\psi_{n+1}(i)}(f_{\bar A}^{2n+1}(T_{i-1}A_i))$$ and $$f_{\bar A}^{2n+1}(T_{i-1}A_i)\in (P_{\psi_{n+1}(i)},A_{\psi_{n+1}(i)}]$$ we have that $f_{\bar A}^{2n+2}(T_iA_i)\in {(Q_{\rho(\psi_{n+1}(i))},Q_{\rho(\psi_{n+1}(i))+1})}$. 
Using relations \eqref{shortcycle2}, \eqref{eq:even}, \eqref{eq:2n3}, the following holds:
\begin{align*}
T_{\rho(\psi_{n+1}(i))}(f_{\bar A}^{2n+2}(T_{i}A_i))&=T_{\rho(\psi_{n+1}(i))}T_{\psi_{n+1}(i)}(f_{\bar A}^{2n+1}(T_{i-1}A_i))\\
&=T_{\theta(\psi_{n+1}(i)-1)}T_{\psi_{n+1}(i)-1}(f_{\bar A}^{2n+1}(T_{i-1}A_i))\\
&=f_{\bar A}^{2n+3}(T_{i-1}A_i).
\end{align*}
For the cycle property not to hold,  one has
$$f_{\bar A}^{2n+3}(T_iA_i)\ne f_{\bar A}^{2n+3}(T_{i-1}A_i)\;(=T_{\rho(\psi_{n+1}(i))}(f_{\bar A}^{2n+2}(T_iA_i))).$$ Hence, 
$$f_{\bar A}^{2n+2}(T_iA_i)\in (Q_{\rho(\psi_{n+1}(i))},Q_{\rho(\psi_{n+1}(i))+1})\setminus (Q_{\rho(\psi_{n+1}(i))},A_{\rho(\psi_{n+1}(i))+1})
$$ and relations \eqref{eq:odd} are proved for $k=n+1$. 

One proceeds similarly to prove \eqref{eq:even} for $k=n+1$.
\end{proof}

We can now complete the proof of Theorem \ref{cycle}. Assume by contradiction that the cycle property does not hold. Thus relations \eqref{eq:odd} and \eqref{eq:even} will be satisfied for all $n$. In particular $f_{\bar A}^{2n+1}(T_{i-1}A_i)\in (P_{\psi_{n+1}(i)},A_{\psi_{n+1}(i)}]$. Recall that $\psi_n(i)=(\theta\circ\rho)^n(i)$. A direct computation shows that $\theta(\rho(i))=4g-4+i \pmod {8g-4}$, so
$$\psi_n(i)=i+n(4g-4) \pmod {8g-4}.$$ We show 
that there exists $n$ such that $\psi_{n}(i)$ belongs to a congruence class of one of the numbers
$\{2,2g+1, 4g, 6g-1\}$. More precisely, 
\begin{enumerate}
\item if $i\equiv 0\pmod 4$, then there exists $n$ such that
\[
\psi_n(i)\equiv 4g\pmod{8g-4};
\]
\item if $i\equiv 2\pmod 4$, then there exists $n$ such that
\[
\psi_n(i)\equiv 2\pmod{8g-4};
\]
\item if $i\equiv 1\pmod 4$ and $g$ is even, then there exists $n$ such that
\[
\psi_n(i)\equiv 2g+1\pmod{8g-4};
\]
if $i\equiv 1\pmod 4$ and $g$ is odd, then there exists $n$ such that
\[
\psi_n(i)\equiv 6g-1\pmod{8g-4};
\]
\item if $i\equiv 3\pmod 4$ and $g$ is even, then there exists $n$ such that
\[
\psi_n(i)\equiv 6g-1\pmod{8g-4};
\]
if $i\equiv 3\pmod 4$ and $g$ is odd, then there exists $n$ such that
\[
\psi_n(i)\equiv 2g+1\pmod{8g-4};
\]
\end{enumerate}
This follows from the fact that for any $g\geq 2$, $g-1$ and $2g-1$ are relatively prime.

We will give a proof of the last 
statement in part (4). Let $i=4k+3$. Then $\psi_n(i)=4k+3+4n(g-1)$. Since $g$ is odd, $2g-2$ is divisible by $4$, i.e. $2g-2=4s$ for some integer $s$. Since $g-1$ and $2g-1$ are relatively prime, there exist integers $n$ and $m$ such that
\[
k+n(g-1)=s+m(2g-1).
\]
Multiplying by $4$ and adding $3$ to both sides, we obtain
\[
3+4k+4n(g-1)=3+4s+4m(2g-1)=2g-2+4m(2g-1)+3,
\]
and therefore
\[
\psi_n(i)\equiv 2g+1\pmod{8g-4}.
\]

Let $n$ be such an integer, with the property that $\psi_n(i)$ belongs to the congruence class of one of the
numbers $\{2, 2g+ 1, 4g, 6g -1\}$. 
By Lemma \ref{lem:period2}, 
$Q_{\psi_{n}(i)}$ is fixed by $T_{\psi_{n}(i)}$. 
Using \eqref{eq:even} we have
$f_{\bar A}^{2n-1}(T_{i-1}A_i)\in (P_{\psi_{n}(i)},A_{\psi_{n}(i)}]$ and 
$$f_{\bar A}^{2n}(T_iA_i)=T_{\psi_{n}(i)}(f_{\bar A}^{2n-1}(T_{i-1}A_i))\in (Q_{\psi_n(i)-1},T_{\psi_{n}(i)}A_{\psi_{n}(i)}]\subset (Q_{\psi_n(i)-1},Q_{\psi_n(i)}).$$
The interval $[A_{\psi_{n}(i)},Q_{\psi_{n}(i)})$ {expands} under $T_{\psi_{n}(i)}$, so
$T_{\psi_{n}(i)}A_{\psi_{n}(i)} \in (Q_{\psi_n(i)-1},A_{\psi_n(i)})$. Therefore,
$f_{\bar A}^{2n}(T_iA_i)\in (Q_{\psi_n(i)-1},A_{\psi_n(i)})$, which assures us that the cycle property holds since
\[f_{\bar A}^{2n+1}(T_iA_i)=T_{\psi_n(i)-1}(f_{\bar A}^{2n}(T_iA_i))=T_{\rho(\psi_{n}(i))}(f_{\bar A}^{2n}(T_iA_i))=f_{\bar A}^{2n+1}(T_{i-1}A_i)\,.\qedhere\]
\end{proof}

\begin{rem} In contrast, if $\bar A=\bar P$ the upper and lower orbits of all $P_i$ are  periodic.
Specifically, 
\[
\begin{aligned}
\OO_u(P_i ) &=\{Q_{\rho(i)+1}, Q_{\rho(i)+1},  \dots\}\text{ if }i\in \{2, 2g+1, 4g, 6g-1\}\\
\OO_u(P_i ) &=\{, Q_{\rho(i)+1}, Q_i, Q_{\rho(i)+1}, Q_i, \dots\}\text{ for other }i,
\end{aligned}
\]
and
\[
\begin{aligned}
\OO_\ell(P_i)&=\{P_i, P_i,\dots\}\text{ if }i\in \{1, 2g, 4g-1, 6g-2\}\\
\OO_\ell(P_i)&=\{P_{\theta(i-1)}, P_i, P_{\theta(i-1)},\dots\}\text{ for other }i.
\end{aligned}
\]
Notice that these two phenomena have something in common: in both cases the sets of values are finite.
\end{rem}

We have seen in the proof of Theorem \ref{cycle} that,
when $T_iA_i\in  (Q_{\rho(i)}, A_{\rho(i)+1})$ and $T_{i-1}A_i\in  (A_{\theta(i-1)}, P_{\theta(i-1)+1})$, the cycle property holds immediately with $m_i=k_i=1$, by using relation \eqref{shortcycle2}. In this case we have
\begin{equation}\label{eq:short}
f_{\bar A}(T_iA_i)=f_{\bar A}(T_{i-1}A_i).
\end{equation}

\begin{defn}A partition point $A_i$ is said to satisfy the {\em short cycle property} if  (\ref{eq:short}) holds, or, equivalently, if
\[
T_iA_i\in  (Q_{\rho(i)}, A_{\rho(i)+1})\text{ and }T_{i-1}A_i\in  (A_{\theta(i-1)}, P_{\theta(i-1)+1}).
\]
\end{defn}
This notion  will be used in the next section.

\begin{rem} \label{aibi} The existence of an open set of partitions $\bar A$ satisfying the short cycle property follows from Corollary \ref{new} of the Appendix: it is sufficient to take $A_i\in (b_i, a_i)$ for each $i$.
\end{rem}

\section{Construction of $\Omega_{\bar A}$}\label{s:bijectivity}
According to the philosophy of the $\sz$ situation treated in \cite{KU3} we expect the $y$-levels of the attractor set of $F_{\bar A}$, $\Omega_{\bar A}$,  to be comprised from the values of the cycles of 
$\{A_i\}$. If the cycles are short,  the situation is rather simple:
$y$-levels of the upper connected component of $\Omega_{\bar A}$ are 
\[
B_i:=T_{\sigma(i-1)}A_{\sigma(i-1)},
\]
and
$y$-levels of the lower connected component of $\Omega_{\bar A}$ are 
\[
{C_i:=T_{\sigma(i+1)}A_{\sigma(i+1)+1}}.
\]
The $x$-levels in this case are the same as for the Bowen-Series map $F_{\bar P}$, and the set $\Omega_{\bar A}$ is determined by the  corner points located in the strip 
\[
\{(x,y\in\Sb\times\Sb\mid y\in[A_i,A_{i+1})\}
\]
(see Figure \ref{fig:strip}) with coordinates
$$(P_{i},B_i) \text{ (upper part )}  \; \text{ and } \;  (Q_{i+1},C_i) \text{ (lower part)}.$$ 
This set  {obviously has} a finite rectangular structure.

\begin{figure}[htb]
\includegraphics[scale=1.1]{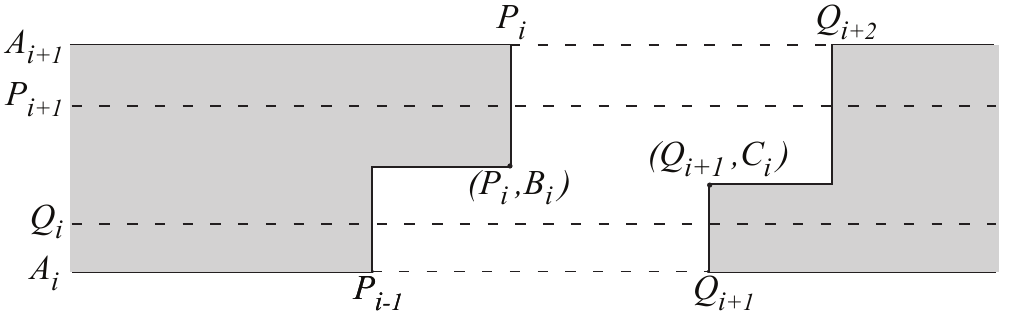}
\caption{Strip $y\in [A_i,A_{i+1}]$ of $\Omega_{\bar A}$}
\label{fig:strip}
\end{figure}

We will prove the desired properties of the set $\Omega_{\bar A}$ stated in Theorem \ref{main}: property (1) (Theorem \ref{thm:bij}) and property (2) (Theorem \ref{th:reduction}).

\begin{rem} Alternatively, the domain of bijectivity of $F_{\bar A}$ can be constructed using an approach first described by of I. Smeets in her thesis \cite{Sm}: start with the known domain $\Omega_{\bar P}$ of the Bowen-Series map $F_{\bar P}$ and modify it by an infinite ``quilting process" by adding and deleting rectangles where the maps $F_{\bar A}$ and $F_{\bar P}$ differ. In the case of short cycles the ``quilting process" gives exactly the region 
$\Omega_{\bar A}$, but unfortunately, it does not work when the cycles are longer. 
Since in the short cycles case the domain $\Omega_{\bar A}$ can be described explicitly, we
do not go into the details of the ``quilting process" here.
\end{rem}

\begin{thm}\label{thm:bij}
The map $F_{\bar A}: \Omega_{\bar A}\ra \Omega_{\bar A}$ is one-to-one and onto.
\end{thm}
\begin{proof}
We investigate how different regions of $\Omega_{\bar A}$ are mapped by $F_{\bar A}$. More precisely we look at the strip $S_i$  of $\Omega_{\bar A}$ given by $y\in [A_i,A_{i+1}]$, and its image under $F_{\bar A}$, in this case $T_i$. See Figure \ref{fig:A-bij}.
We consider the following decomposition of this strip: $\tilde S_i=[Q_{i+2},P_{i-1}]\times [A_i,A_{i+1}]$ (red rectangular piece),
$S_i^\ell=[Q_{i+1},Q_{i+2}]\times [A_i, {C_i}]$ (blue lower corner) and $S_i^u=[P_{i-1},P_i]\times
[{B_i},A_{i+1}]$
 (green upper corner). Now
\begin{align}
T_i(\tilde S_i)& =T_i([Q_{i+2},P_{i-1}]\times [A_i,A_{i+1}])=[Q_{\sigma(i)},P_{\sigma(i)+1}]\times [B_{\sigma(i)+1},C_{\sigma(i)-1}]\\
T_i(S_i^\ell) & = T_i( [Q_{i+1},Q_{i+2}]\times [A_i, {C_i}])=[P_{\sigma(i)},Q_{\sigma(i)}]\times [B_{\sigma(i)+1}, T_i{C_i}]\\
T_i(S_i^u) & =T_i([P_{i-1},P_i]\times
[{B_i},A_{i+1}])=[P_{\sigma(i)+1},Q_{\sigma(i)+1}]\times [T_i{B_i},C_{\sigma(i)-1}]
\end{align}

\begin{figure}[htb]
\includegraphics[scale=1]{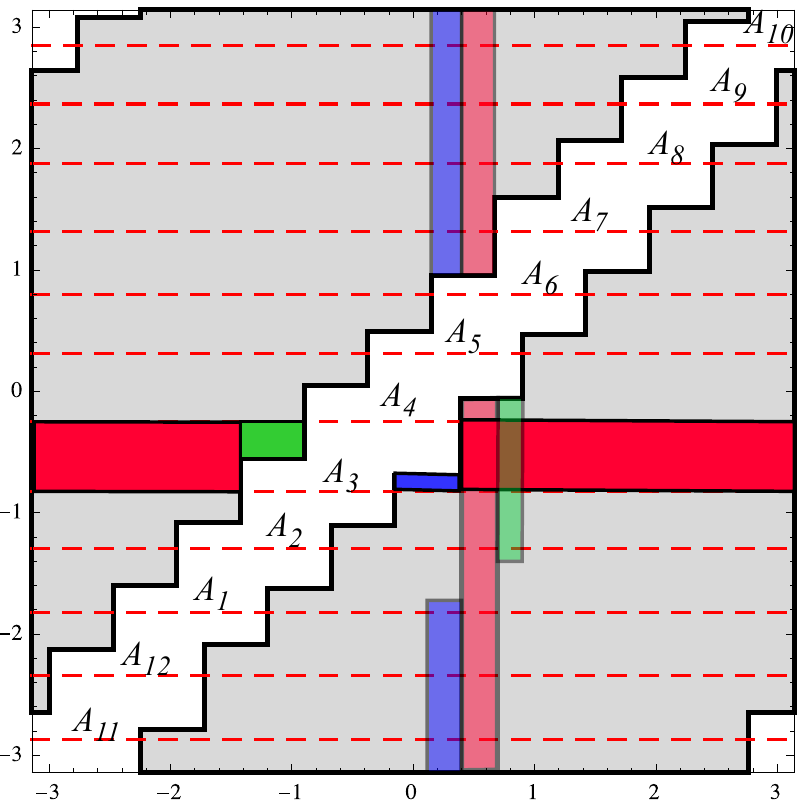}
\caption{Bijectivity of the $F_{\bar A}$ map}
\label{fig:A-bij}
\end{figure}

Notice that 
\begin{itemize}
\item $T_i(\tilde S_i)$ is a complete vertical strip in $\Omega_{\bar A}$, $Q_{\sigma(i)}\le x\le P_{\sigma(i)+1}$;
\item $T_i(S_i^u)$ together with $T_j(S_j^\ell)$ (where ${\sigma(j+1)=\sigma(i-1)-1}$) form a complete vertical strip in  $\Omega_{\bar A}$, $P_{\sigma(i)+1}\le x\le Q_{\sigma(i)+1}$. (We are using here the short cycle property $T_iT_{\sigma(i-1)}A_{\sigma(i-1)}=T_jT_{\sigma(j+1)}A_{\sigma(j+1)+1}$.)
\item  $T_i(S_i^\ell)$ together with $T_k(S_k^u)$ (where $\sigma(k)+1=\sigma(i)$) form a complete vertical strip in  $\Omega_{\bar A}$, $P_{\sigma(i)}\le x\le Q_{\sigma(i)}$.
\end{itemize}
This proves the bijectivity property of $F_{\bar A}$ on $\Omega_{\bar A}$.
\end{proof}

{We showed that  the ends of the cycles do not appear as $y$-levels of the boundary of $\Omega_{\bar A}$. We state this important property as a corollary.}
\begin{cor} For $i$ and $j$ related via  $\sigma(j+1)=\sigma(i-1)-1$, we have
\begin{equation}\label{eq:ij}
T_jC_j=T_iB_i\in[B_{\rho(i)+1},C_{\theta(i)}]=[B_{\rho(j)},C_{\theta(j)-1}].
\end{equation}
\end{cor}

\section{Trapping region}\label{s:trapping}
 In order to prove property (2) of  $\Omega_{\bar A}$, we enlarge it and prove the trapping property for the enlarged region first.
Let
$\Psi_{\bar A}=\Omega_{\bar A}\cup \mathcal D$, where \[\mathcal D=\bigcup_{i=1}^{8g-4} R_i  \text{ and }
R_i=[P_{i-1},P_i]\times[Q_i,B_i].\]
Notice that $\Psi_{\bar A}$ can be also expressed as $\Psi_{\bar A}=\Omega_{\bar P}\cup \mathcal A$, where
 $\mathcal A=\cup_{i=1}^{8g-4} [Q_{i+1},Q_{i+2}]\times[P_i,C_i]$. The $y$-levels of the upper part of $\Psi_{\bar A}$ are given by the $Q_i$'s and the   $y$-levels of the lower part of $\Psi_{\bar A}$ are given by the $C_i$'s.
 
 \begin{figure}[htb]
\includegraphics[scale=1]{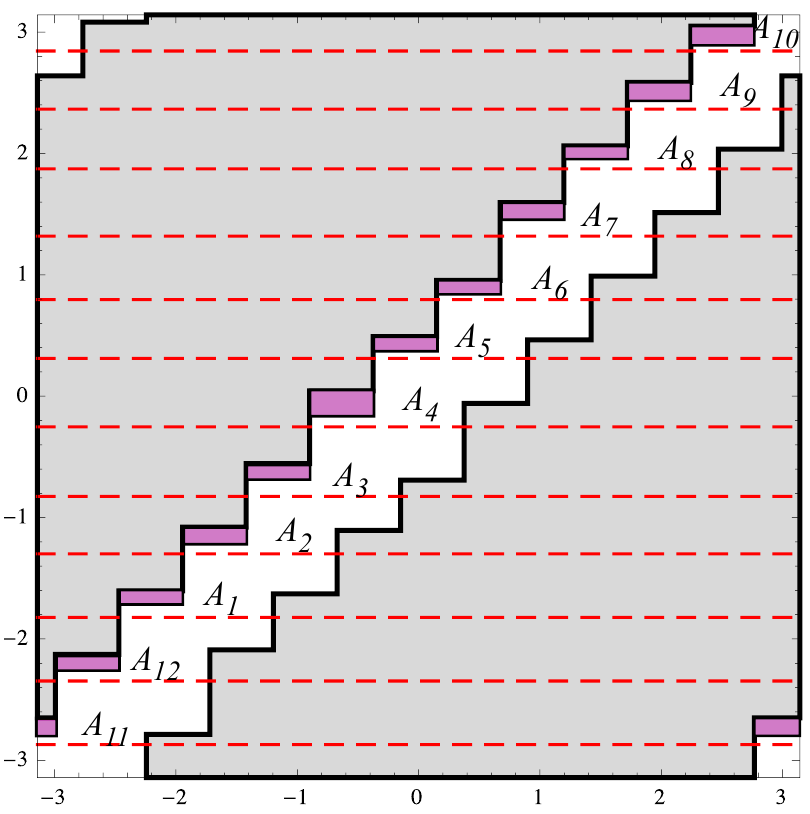}
\caption{Trapping region $\Psi_{\bar A}$ consisting of the set $\Omega_{\bar A}$ (grey) and the added set  $\mathcal D$ (purple)} 
\label{fig:purple}
\end{figure}

\begin{thm}\label{thm:trap}
The set $\Psi_{\bar A}$ is a trapping region for the map $F_{\bar A}$, i.e.,
\begin{itemize}
\item given any $(x,y) \in  \Sb\times \Sb\setminus \Delta$, there exists $n\ge 0$ such that  $F_{\bar A}^n(x,y)\in \Psi_{\bar A}$;
\item $F_{\bar A}(\Psi_{\bar A})\subset \Psi_{\bar A}$.
\end{itemize}
\end{thm}
\begin{proof}
 We start with $(x,y)\in \Sb\times \Sb\setminus \Delta$ and show that there exists $n\ge 0$ such that
$F_{\bar A}^n(x,y)\in \Psi_{\bar A}$. 
We have $Q_i\in  [A_i,P_{i+1})\subset [A_i,A_{i+1})$, and by the short cycle condition, 
${C_i}\in [A_i,P_{i+1})\subset [A_i,A_{i+1})$.

Consider $(x,y)\in \Sb\times \Sb\setminus \Delta$. Notice that there exists $n(x,y)>0$ such that the two values $x_n,y_n$ obtained from the $n$th iterate of $F_{\bar A}$,  $(x_n,y_n)=F^n_{\bar A}(x,y)$, are not inside the same isometric circle; in other words, $(x_n,y_n)\not\in X_i=[P_{i},Q_{i+1}]\times [A_{i},A_{i+1})$ for all  $1\le i\le 8g-4$ (see the argument in the proof of Theorem \ref{thm:BS}).

In order to prove the attracting property we need to analyze the situations $(x_n,y_n)\in Y_i=[P_{i-1},P_{i}] \times [A_{i},Q_i)$ (orange set), and $(x_n,y_n)\in Z_i=[Q_{i+1},Q_{i+2}]\times({C_i},A_{i+1}]$ (green set),  and show that a forward iterate lands in $\Psi_{\bar A}$.

\begin{figure}[htb]
\includegraphics[scale=1.2]{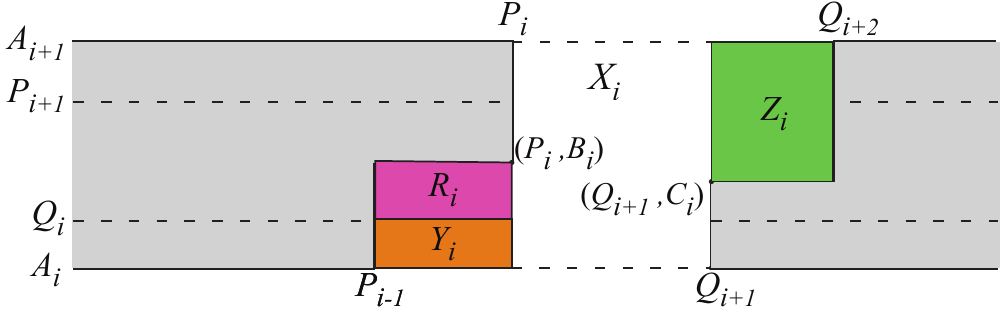}
\caption{The strip $y\in [A_i,A_{i+1}]$ of the trapping region $\Psi_{\bar A}$ together with the sets $Y_i=[P_{i-1},P_{i}] \times [A_{i},Q_i)$ (orange) and $Z_i=[Q_{i+1},Q_{i+2}]\times({C_i},A_{i+1}]$ (green) outside of it that require special considerations}
\label{fig:trap}
\end{figure}

\medskip

\noindent \textbf{Case (I)} If $(x_n,y_n)\in Y_i=[P_{i-1},P_{i}]\times [A_{i},Q_i)$, then\[
F_{\bar A}(x_n,y_n)\in [T_iP_{i-1},T_iP_i]\times [T_iA_i,T_iQ_i)=[P_{\rho(i)},Q_{\rho(i)}]\times[B_{\rho(i)},Q_{\rho(i)+1}).
\]
Since $B_{\rho(i)}\in [Q_{\rho(i)},A_{\rho(i)+1}]$, we need to analyze the regions
\[[P_{\rho(i)},Q_{\rho(i)}]\times[B_{\rho(i)},A_{\rho(i)+1}] \quad \text{ and }\quad [P_{\rho(i)},Q_{\rho(i)}]\times[A_{\rho(i)+1},Q_{\rho(i)+1}).\]

\smallskip

\noindent \textbf{(a)} If $(x_{n+1},y_{n+1})\in [P_{k},Q_{k}]\times[B_k,A_{k+1}]$, where $k=\rho(i)$, then 
\[
(x_{n+2},y_{n+2})=T_k(x_{n+1},y_{n+1})\in [Q_{\rho(k)},Q_{\rho(k)+1}]\times[T_kB_k,T_kA_{k+1}]\,.
\]
Since $T_kA_{k+1}=C_{\theta(k)}$, and $T_kB_k\in [B_{\rho(k)+1},C_{\theta(k)}]$ the only part of the vertical strip above where $(x_{n+2},y_{n+2})$ might still lie outside of $\Psi_{\bar A}$ is a subset of $[P_{\rho(k)+1},Q_{\rho(k)+1}]\times[B_{\rho(k)+1},Q_{\rho(k)+2})$.

Notice that $\rho(k)=\sigma(\sigma(i)+1)+1=4g+i-2$ (direct verification), so we need to analyze the situation $(x_{n+2},y_{n+2})\in [P_{4g+i-1},Q_{4g+i-1}]\times [B_{4g+i-1},Q_{4g+i})$.

\smallskip

\noindent \textbf{(b)} If  $(x_{n+1},y_{n+1})\in [P_k,Q_k]\times[A_{k+1},Q_{k+1})$, then
\[
(x_{n+2},y_{n+2})=T_{k+1}(x_{n+1},y_{n+1})\in [P_{\rho(k+1)},T_{k+1}Q_{k}]\times[B_{\rho(k+1)},Q_{\rho(k+1)+1})\,.
\]
Notice that $T_{k+1}Q_k\in [P_{\rho(k+1)},Q_{\rho(k+1)}]$  and $\rho(k+1)=\rho(\rho(i)+1)=i-1$ (direct verification) so we are left to investigate $(x_{n+2},y_{n+2})\in [P_{i-1},Q_{i-1}]\times [B_{i-1},Q_{i})$.

\medskip

To summarize, we started with $(x_{n+1},y_{n+1})\in [P_{\rho(i)},Q_{\rho(i)}]\times[B_{\rho(i)},Q_{\rho(i)+1})\,$ and found two situations that need to be analyzed $(x_{n+2},y_{n+2})\in [P_{i-1},Q_{i-1}]\times [B_{i-1},Q_{i})$ or  $(x_{n+2},y_{n+2})\in [P_{4g+i-1},Q_{4g+i-1}]\times [B_{4g+i-1},Q_{4g+i})$.

We prove in what follows that it is not possible for all future iterates $F^m(x_n,y_n)$  to belong to the sets of type $[P_k,Q_k]\times[B_{k},Q_{k+1})$.

 First, it is not possible for all $F^m(x_n,y_n)$ (starting with some $m>0$) to belong only to type-a sets 
$[P_{k_m},Q_{k_m}]\times[B_{k_m},A_{k_m+1}]$, where $k_{m+1}=\rho(k_m)+1$ because such a set is included in the isometric circle $X_{k_m}$, and the argument at the beginning of the proof disallows such a situation.

Also, it is not possible for all $F^m(x_n,y_n)$ (starting with some $m>0$) to belong only to type-b sets $[P_{k_m},Q_{k_m}]\times[A_{k_m+1},Q_{k_m+1})$, where $k_{m+1}=\rho(k_{m}+1)$: this would imply that the pairs of points $(y_{n+m},A_{k_{n+m}+1})$ (on the y-axis) will belong to the same interval $[A_{k_{n+m}},Q_{k_{n+m}+1})$ which is impossible due to expansiveness property of the map $f_{\bar A}$.

Therefore, there exists a pair  $(x_{l},y_{l})$ in the orbit of $F^m(x_n,y_n)$ such that 
$$(x_{l},y_{l})\in  [P_{j},Q_{j}]\times[A_{j+1},Q_{j+1}) \quad \text{ (type-b)}$$ for some $1\le j\le 8g-4$ and $$(x_{l+1},y_{l+1})\in [P_{j'},T_{j+1}Q_{j}]\times[T_{j+1}A_{j+1},P_{j'+1}] \subset [P_{j'},Q_{j'}]\times[Q_{j'},P_{j'+1}] \quad \text{ (type-a)},$$ where $j'=\rho(j+1)$. Then 
$$(x_{l+2},y_{l+2})\in T_{j'}([P_{j'},T_{j+1}Q_{j}] \times[Q_{j'},P_{j'+1}])=[Q_{j''},T_{j'}T_{j+1}Q_{j}] \times[Q_{j''+1},P_{j''-2}]$$
where $j''=\rho(j')$. 

Using the results of the Appendix (Corollary \ref{lem:app}), we have that the arc length distance   satisfies
$$\ell(P_{j'},T_{j+1}Q_j)=\ell(T_{j+1}P_j,T_{j+1}Q_j)<\frac{1}{2}\ell(P_{j'},Q_{j'}).$$
Now we can use Corollary \ref{new} (ii) applied to the point $T_{j+1}Q_j\in[P_{j'},Q_{j'}]$ to conclude that  $T_{j'}T_{j+1}Q_j\in [Q_{j''},P_{j''+1}]$. 
Therefore $(x_{l+2},y_{l+2})\in \Psi_{\bar A}$.

\bigskip

\noindent \textbf{Case (II)} If $(x_n,y_n)\in Z_i=[Q_{i+1},Q_{i+2}]\times(C_i,A_{i+1}]$,
then 
\[
F_{\bar A}(x_n,y_n)\in T_i ( [Q_{i+1},Q_{i+2}]\times (C_i,A_{i+1}])=[P_{\sigma(i)},Q_{\sigma(i)}]\times (T_iC_i,C_{\theta(i)}]. 
\]
Since
\(
T_iC_i\in [B_{\rho(i)},C_{\theta(i)-1}]
\)
by \eqref{eq:ij} and the set $[P_{\sigma(i)},Q_{\sigma(i)}]\times [B_{\rho(i)},C_{\theta(i)-1}]$ is in $\Psi_{\bar A}$, we are left with analyzing the situation 
$$(x_{n+1},y_{n+1})\in [P_{\sigma(i)},Q_{\sigma(i)}]\times (C_{\theta(i)-1},C_{\theta(i)}].
$$
This requires two subcases depending on $y_{n+1}\in  (C_{k-1}, A_{k})$ or  $y_{n+1}\in [A_{k}, C_{k}]$, where $k=\theta(i)$.

\smallskip

\noindent \textbf{(a)} If $(x_{n+1},y_{n+1})\in [P_{k+1},Q_{k+1}]\times (C_{k-1}, A_{k})$, then
$$(x_{n+2},y_{n+2})=T_{k-1}(x_{n+1},y_{n+1})\in [T_{k-1}P_{k+1},Q_{\sigma(k-1)}]\times (T_{k-1}C_{k-1}, T_{k-1}A_{k}).$$
Notice that $\sigma(k-1)=\sigma(\theta(i)-1)=i+2$ (direct verification). Since \[T_{k-1}P_{k+1}\in [P_{\sigma(k-1)},Q_{\sigma(k-1)})=[P_{i+2},Q_{i+2}),\] $T_{k-1}A_k=C_{\theta(k-1)}=C_{i+1}$ and $T_{k-1}C_{k-1}\in[B_{\rho(k-1)},C_{\theta(k-1)-1})=[B_{i+3},C_{i})$, we have that $(x_{n+2},y_{n+2})\in [P_{i+2},Q_{i+2})\times [B_{i+3},C_{i+1})$. The only part of this vertical strip where $(x_{n+2},y_{n+2})$ might still lie outside of $\Psi_{\bar A}$ is a subset of $[P_{i+2},Q_{i+2}]\times(C_i,C_{i+1})$, and that is the situation we need to analyze.
\smallskip

\noindent \textbf{(b)} If  $(x_{n+1},y_{n+1})\in [P_{k+1},Q_{k+1}]\times[A_{k},C_{k}]$, then
\[
\begin{aligned}
(x_{n+2},y_{n+2})\in T_{k}\left ( [P_{k+1},Q_{k+1}]\times[A_{k},C_{k}]\right)
=[P_{\sigma(k)-1},P_{\sigma({k})}]\times[B_{\rho(k)},T_{k}C_{k}].
\end{aligned}
\]
Since $T_kC_k\in [B_{\rho(k)},C_{\theta(k)-1}]$ by \eqref{eq:ij} and $\sigma(k)=\sigma(\theta(i))=4g+i-1$, then $$(x_{n+2},y_{n+2})\in [P_{4g+i-2},P_{4g+i-1}]\times [B_{4g+i},C_{4g+i-3}]$$ and the only part of this vertical strip where $(x_{n+2},y_{n+2})$ might still lie outside of $\Psi_{\bar A}$ is a subset of $[P_{4g+i-2},Q_{4g+i-2}]\times [A_{4g+i-3},C_{4g+i-3}]$.

\smallskip

To summarize, we started with $(x_{n+1},y_{n+1})\in [P_{\sigma(i)},Q_{\sigma(i)}]\times (C_{\theta(i)-1},C_{\theta(i)}]$ and found two situations that need to be analyzed $(x_{n+2},y_{n+2})\in [P_{i+2},Q_{i+2}]\times (C_{i},C_{i+1}]$ or  $(x_{n+2},y_{n+2})\in [P_{4g+i-2},Q_{4g+i-2}]\times [A_{4g+i-3},C_{4g+i-3}]$.

\smallskip

We prove that it is not possible for all future iterates $F^m(x_n,y_n)$  to belong to the sets of type $[P_{k+1},Q_{k+1}]\times[C_{k-1},C_{k}]$.

First, it is not possible for all $F^m(x_n,y_n)$ (starting with some $m>0$) to belong only to type-a sets $[P_{k_m+1},Q_{k_m+1}]\times(C_{k_m-1},A_{k_m})$, where $k_{m+1}=\sigma(k_m-1)$: this would imply that the pairs of points $(y_{n+m},A_{k_{n+m}})$ (on the $y$-axis) will belong to the same interval $(C_{k_{n+m}-1},A_{k_{n+m}}]\subset [A_{k_{n+m}-1},A_{k_{n+m}}]$ which is impossible due to expansiveness property of the map $f_{\bar A}$ on such intervals.

From the discussion of Case (b), if an iterate $F^m(x_n,y_n)$ belongs to a type-b set, then $F^{m+1} (x_n,y_n)$ either belongs to $\Psi_{\bar A}$ or to another type-b set. However, it is not possible for all iterates $F^m(x_n,y_n)$ (starting with some $m>0$) to belong to type-b sets 
$[P_{k_m+1},Q_{k_m+1}]\times[A_{k_m},C_{k_m}]$, where $k_{m+1}=\sigma(k_m)-2$ because such a set is included in the isometric circle $X_{k_m}$, and the argument at the beginning of the proof disallows such a situation. Thus, once an iterate $F^m(x_n,y_n)$ belongs to a type-b set, then it will eventually belong to $\Psi_{\bar A}$.

We showed that any point $(x,y)$ that belongs to a set $[P_{k+1},Q_{k+1}]\times (C_{k-1},C_k]$ will  have a future iterate in $\Psi_{\bar A}$. This completes the proof of Case II and, hence, the theorem.
\end{proof}

\section{Reduction theory}\label{s:reduction} We can now complete the proof of Theorem \ref{main}. 
\begin{thm}\label{th:reduction} For almost every point $(x,y)\in \Sb\times \Sb\setminus \Delta$, there exists $K>0$ such that $F_{\bar A}^K(x,y)\in \Omega_{\bar A}$, and the set $\Omega_{A}$ is a global attractor for $F_{\bar A}$, i.e., $$\Omega_{A}=\bigcap_{n=0}^\infty F_{\bar A}^n(\Sb\times \Sb\setminus \Delta).$$
\end{thm}
\begin{proof} 

By Theorem \ref{thm:trap}, every point $(x,y)\in \Sb\times \Sb\setminus \Delta$ is mapped to the trapping region $\Psi_{\bar A}=\Omega_{\bar A}\cup \mathcal D$ by some iterate $F_{\bar A}^n$. Therefore, it suffices to track the set $\mathcal D=\bigcup_{i=1}^{8g-4} R_i$. The image of each rectangle $R_i=[P_{i-1}, P_i]\times[Q_i,B_i]$ under $F_ {\bar A}$, $F_{\bar A}(R_i)=T_i(R_i)$, is a rectangular set
\begin{equation}\label{eq:TD1}
F_{\bar A}(R_i)=[T_iP_{i-1}, T_iP_i]\times[T_iQ_i, T_iB_i]=[P_{\rho(i)}, Q_{\rho(i)}]\times[Q_{\rho(i)+1}, T_iB_i].
\end{equation}
The ``top" of this rectangle, $[P_{\rho(i)}, Q_{\rho(i)}]\times \{T_iB_i\}$ is inside $\Omega_{\bar A}$, since $T_iB_i\in [B_{\rho(i)+1},C_{\theta(i)}]$.
Moreover, 
\begin{equation}\label{eq:TD2}
F_{\bar A}(R_i)\setminus \Omega_{\bar A}=[P_{\rho(i)}, Q_{\rho(i)}]\times[Q_{\rho(i)+1}, B_{\rho(i)+1}]\subset R_{\rho(i)+1},
\end{equation}
so, by letting $j=\rho(i)+1$,
\[
F_{\bar A}(\mathcal D)\setminus \Omega_{\bar A}=\bigcup_{j=1}^{8g-4}[P_{j-1}, Q_{j-1}]\times[Q_j,B_j]
\]
and
\[F_{\bar A}(\Omega_{\bar A}\cup \mathcal D)=\Omega_{\bar A}\cup \bigcup_{j=1}^{8g-4}[P_{j-1}, Q_{j-1}]\times[Q_j,B_j].\]
Now the image of the rectangular set $[P_{j-1}, Q_{j-1}]\times[Q_j,B_j]$ under $F_{\bar A}(=T_j)$ is
\[
F_{\bar A}([P_{j-1}, Q_{j-1}]\times[Q_j,B_j])=[P_{\rho(j)},T_jQ_{j-1}]\times[Q_{\rho(j)+1}, T_jB_j],
\] hence
\begin{equation}\label{eq:fa}
F_{\bar A}\big(F_{\bar A}(\mathcal D))\setminus \Omega_{\bar A}=\bigcup_{j=1}^{8g-4}[P_{\rho(j)},T_jQ_{j-1}]\times[Q_{\rho(j)+1}, B_{\rho(j)+1}].
\end{equation}
Corollary \ref{lem:app} tells us that the length of the segment $[P_{\rho(j)},T_jQ_{j-1}]=T_j([P_{j-1},Q_{j-1}])$ is less than $\frac{1}{2}$ of $[P_{\rho(j)},Q_{\rho(j)}]$. If we let $k=\rho(j)+1$, and denote $T_jQ_{j-1}$ by $S^{(2)}_k$, then \eqref{eq:fa} becomes
\[
F^{2}_{\bar A}(\mathcal D)\setminus \Omega_{\bar A}=\bigcup_{k=1}^{8g-4}[P_{k-1},S^{(2)}_{k-1}]\times[Q_k, B_k] 
\]
with the length of the segment $[P_{k-1},S^{(2)}_{k-1}]$ being less than $\frac{1}{2}$ of $[P_{k-1},Q_{k-1}]$. Inductively, it follows that:
\begin{equation}\label{eq:fn}
F^{n}_{\bar A}(\mathcal D)\setminus \Omega_{\bar A}=\bigcup_{k=1}^{8g-4}[P_{k-1},S^{(n)}_{k-1}]\times[Q_k, B_k] 
\end{equation}
where the length of the segment $[P_{k-1},S^{(n)}_{k-1}]$ is less than $\frac{1}{2^{n-1}}$ of $[P_{k-1},Q_{k-1}]$.  Thus, 
\[
F^{n}_{\bar A}(\Omega_{\bar A}\cup \mathcal D)=\Omega_{\bar A}\cup \bigcup_{k=1}^{8g-4}[P_{k-1},S^{(n)}_{k-1}]\times[Q_k, B_k] 
\]
and
\begin{align*}
\bigcap_{n=0}^\infty F_{\bar A}^n(\Sb\times \Sb\setminus \Delta)&=\bigcap_{n=0}^\infty F_{\bar A}^n(\Omega_A\cup \mathcal D)=\Omega_{\bar A}\cup\bigcap_{n=0}^\infty \left(\bigcup_{k=1}^{8g-4}[P_{k-1},S^n_{k-1}]\times[Q_k, B_k]\right)\\
&=\Omega_{\bar A}\cup \bigcup_{k=1}^{8g-4}\{P_{k-1}\}\times[Q_k, B_k]=\Omega_{\bar A}
\end{align*}

In what follows, we will show that any point $(x,y)\in \mathcal D$ (see Figure \ref{fig:purple}) is actually mapped to $\Omega_{\bar A}$ after finitely many iterations with the exception of the Lebesgue measure zero set consisting of the union of horizontal segments $\bigcup_{i=1}^{8g-4} [P_{i-1}, P_{i}]\times \{Q_i\}$ and their preimages. For that, let $(x,y)\in R_i$ with $y\ne Q_i$ 
and assume that $F_{\bar A}^n(x,y)=(x_n,y_n)\in F^{n}_{\bar A}(\mathcal D)\setminus \Omega_{\bar A}$. Using \eqref{eq:fn}, this means that the sequence of points $y_n\in (Q_{k_n},B_{k_n}]$ for all $n\ge 1$. But $y_{n+1}=T_{k_n}y_n$, $Q_{k_n+1}=T_{k_n}Q_{k_n}$ and the map $T_{k_n}$ is (uniformly) expanding on $[Q_{k_n},B_{k_n}]$ (a subset of the isometric circle of $T_{k_n}$), which contradicts the assumption  $y_n\in (Q_{k_n},B_{k_n}]$.
\end{proof}

\section{Invariant measures}\label{s:entropy}

It is a standard computation that the measure $d\nu=\displaystyle\frac{|dx|\,|dy|}{|x-y|^2}$ is preserved by M\"obius transformations applied to unit circle variables $x$ and $y$, hence by $F_{\bar A}$. 
Therefore, $F_{\bar A}$ preserves the smooth probability measure
\begin{equation}\label{dnu}
d\nu_{\bar A}=\frac{1}{K_{\bar A}}d\nu, \text{ where } K_{\bar A}=\int_{\Omega_{\bar A}}d\nu.
\end{equation}
Alternatively, by considering $F_{\bar A}$ as a reduction map acting on geodesics, the invariant measure can be derived more elegantly by using the geodesic flow on the hyperbolic disk and the Poincar\'e cross-section maps, but we are not pursuing that direction here.

In what follows, we compute $K_{\bar A}$ for the  case when $\bar A$ satisfies the short cycle property. Recall that the domain $\Omega_{\bar A}$ was described in the proof of Theorem \ref{thm:bij} as:
\begin{equation}\label{eq:oma}
\Omega_{\bar A}=\bigcup_{i=1}^{8g-4}[Q_{i+2},P_{i-1}]\times[A_i,A_i+1]\cup [Q_{i+1},Q_{i+2}]\times [A_i,C_i]\cup [P_{i-1},P_{i}]\times[B_i,A_{i+1}]. 
\end{equation}
\begin{prop}\label{prop:KA}
If the points $A_i$ satisfy the short cycle property and $p_i,q_i,b_i,c_i$ represent the angular coordinates of $P_i$, $Q_i$, $B_i=T_iA_i$, and $C_i=T_{i-1}A_i$, respectively, then
\begin{equation}\label{eq:KA}
\nu(\Omega_A)=K_A=\ln\prod_{i=1}^{8g-4}\frac{|\sin\left(\frac{c_i-q_{i+2}}2\right)| |\sin\left(\frac{b_i-p_{i-1}}2\right)|}{|\sin\left(\frac{b_i-p_i}2\right)| |\sin\left(\frac{c_i-q_{i+1}}2\right)|}.
\end{equation}
\end{prop}
\begin{proof}
Since $\Omega_A$ is given by \eqref{eq:oma}, we have
\[
K_{\bar A}=\int_{\Omega_{\bar A}}d\nu=\sum_{i=1}^{8g-4}\left(\int_{Q_{i+2}}^{P_{i-1}}\int_{A_i}^{A_{i+1}}d\nu+\int_{Q_{i+1}}^{Q_{i+2}}\int_{A_i}^{C_i}d\nu+\int_{P_{i-1}}^{P_{i}}\int_{B_i}^{A_{i+1}}d\nu\right).
\]
In order to compute each of the three integrals above, we use angular coordinates $\theta$ and $\phi$ corresponding to $x=e^{i\theta}$, $y=e^{i\phi}$, and write for some arbitrary values $A,B,C,D$:
\begin{align*}
I_{A,B,C,D}:=\displaystyle\int_{A}^{B}\int_{C}^{D}\frac{|dx||dy|}{|x-y|^2}&=\displaystyle\int_{a}^{b}\int_{c}^{d}\frac{d\theta d\phi}{|\exp(i\theta)-\exp(i\phi)|^2}\\&=\int_{a}^{b}\int_{c}^{d}\frac{d\theta d\phi}{2-2\cos(\theta-\phi)}=:I_{a,b,c,d},
\end{align*}
where $a,b,c,d$ are the angular coordinates corresponding to $A,B,C,D$: 
\[
A=e^{ia}, B=e^{ib}, C=e^{ic}, D=e^{id}.
\]
The double integral (which we denoted by $I_{a,b,c,d}$) can be computed explicitly. First
\begin{equation}\label{ab}
\int_{a}^{b}\frac{d\theta}{2-2\cos(\theta-\phi)}=-\frac12\cot\left(\frac{\theta-\phi}2\right)\biggr\vert_{\theta=a}^{\theta=b}=\frac12\left(\cot\left(\frac{a-\phi}2\right)-\cot\left(\frac{b-\phi}2\right)\right).
\end{equation}
Then, using the fact that the antiderivative $\int \cot x dx=\ln|\sin x|$ we obtain
\begin{align*}
I_{a,b,c,d}&=\frac12\int_c^d\left(\cot\left(\frac{a-\phi}2\right)-\cot\left(\frac{b-\phi}2\right)\right)d\phi\\
&=\left(\ln\left|\sin\left(\frac{\phi-b}2\right)\right|-\ln\left|\sin\left(\frac{\phi-a}2\right)\right|\right)\biggr\vert_{\phi=c}^{\phi=d}\\
&=\ln\left|\sin\left(\frac{d-b}2\right)\right|+\ln\left|\sin\left(\frac{c-a}2\right)\right|-\ln\left|\sin\left(\frac{c-b}2\right)\right|-\ln\left|\sin\left(\frac{d-a}2\right)\right|\\
&=\ln\frac{|\sin\left(\frac{d-b}2\right)| |\sin\left(\frac{c-a}2\right)|}{|\sin\left(\frac{c-b}2\right)| |\sin\left(\frac{d-a}2\right)|}.
\end{align*}
Now, using the angular coordinates $p_i,q_i,a_i,b_i,c_i$ corresponding to the points $P_i$, $Q_i$, $A_i$, $B_i$, $C_i$, we obtain 
\begin{align*}
K_{\bar A}&=\sum_{i=1}^{8g-4}(I_{q_{i+2},p_{i-1},a_i,a_{i+1}}+I_{q_{i+1}q_{i+2},a_i,c_i}+I_{p_{i-1},p_i,b_i,a_{i+1}})\\
&=\ln\prod_{i=1}^{8g-4}\frac{|\sin\left(\frac{a_{i+1}-p_{i-1}}2\right)| |\sin\left(\frac{a_i-q_{i+2}}2\right)|}{|\sin\left(\frac{a_i-p_{i-1}}2\right)| |\sin\left(\frac{a_{i+1}-q_{i+2}}2\right)|}
+\ln\prod_{i=1}^{8g-4}\frac{|\sin\left(\frac{c_i-q_{i+2}}2\right)| |\sin\left(\frac{a_i-q_{i+1}}2\right)|}{|\sin\left(\frac{a_i-q_{i+2}}2\right)| |\sin\left(\frac{c_i-q_{i+1}}2\right)|}\\
&\quad+\ln\prod_{i=1}^{8g-4}\frac{|\sin\left(\frac{a_{i+1}-p_i}2\right)| |\sin\left(\frac{b_i-p_{i-1}}2\right)|}{|\sin\left(\frac{b_i-p_i}2\right)| |\sin\left(\frac{a_{i+1}-p_{i-1}}2\right)|}\\
&=\ln\prod_{i=1}^{8g-4}\frac{|\sin\left(\frac{c_i-q_{i+2}}2\right)| |\sin\left(\frac{b_i-p_{i-1}}2\right)|}{|\sin\left(\frac{b_i-p_i}2\right)| |\sin\left(\frac{c_i-q_{i+1}}2\right)|}.
\end{align*}
The last equality is obtained due to cancellations.
\end{proof}

The circle map $f_{\bar A}$ is a factor of $F_{\bar A}$ 
(projecting on the $y$-coordinate), so one can obtain its smooth 
invariant probability measure $d\mu_{\bar A}$ by integrating $d\nu_{\bar A}$ over $
\Omega_{\bar A}$ with respect to the $u$-coordinate. Thus, from the exact shape of the set $\Omega_{\bar A}$, we  
can calculate the invariant measure precisely.
\begin{prop} $\displaystyle d\mu_{\bar A}=\frac{1}{K_{\bar A}}\sum_{i=1}^{8g-4}\left(
\cot\left(\frac{q_{i+1}-\phi}2\right)-\cot\left(\frac{p_{i}-\phi}2\right)
\right)d\phi.$
\end{prop}
\begin{proof}
\begin{equation*}
d\mu_{\bar A}=\frac{1}{K_{\bar A}}\sum_{i=1}^{8g-4}\left(\int_{Q_{i+2}}^{P_{i-1}}\frac{|dx|}{|x-y|^2}+\int_{Q_{i+1}}^{Q_{i+2}}\frac{|dx|}{|x-y|^2}+\int_{P_{i-1}}^{P_{i}}\frac{|dx|}{|x-y|^2}\right)|dy|.
\end{equation*}
Using the calculations (\ref{ab}) we obtain
\begin{align*}
 d\mu_{\bar A}&=\frac{1}{K_{\bar A}}\sum_{i=1}^{8g-4}\Bigg(
\cot\left(\frac{q_{i+2}-\phi}{2}\right)-\cot\left(\frac{p_{i-1}-\phi}{2}\right)\\
&\qquad\qquad\quad+\cot\left(\frac{q_{i+1}-\phi}2\right)-\cot\left(\frac{q_{i+2}-\phi}2\right)\\
&\qquad\qquad\quad+\cot\left(\frac{p_{i-1}-\phi}2\right)-\cot\left(\frac{p_i-\phi}2\right)\Bigg)d\phi\\
&=\frac{1}{K_{\bar A}}\sum_{i=1}^{8g-4}\left(
\cot\left(\frac{q_{i+1}-\phi}2\right)-\cot\left(\frac{p_{i}-\phi}2\right)
\right)d\phi.\qedhere
\end{align*}
\end{proof}

\section{Appendix}
In this section we use the explicit description of the fundamental domain $\F$ given in the Introduction to obtain certain estimates used in the proofs.

The fundamental domain $\F$ is a regular $(8g-4)$-gon bounded by  the isometric circles of the generating transformations of $\G$ with all internal angles equal to $\frac{\pi}2$.  Let us label the vertices of $\F$ by $V_1,\dots, V_{8g-4}$, where $V_i$ is the intersection of the geodesics $P_{i-1}Q_i$ and $P_iQ_{i+1}$ (see Figure \ref{anglecalc} for $g=3$). 
We first prove the following geometric lemma.

\begin{lem} \label{geometric} Consider five consecutive  isometric circles of $\F$: $P_{i-2}Q_{i-1}$, $P_{i-1}Q_i$, $P_iQ_{i+1}$, $P_{i+1}Q_{i+2}$, and $P_{i+2}Q_{i+3}$. Then 
\begin{enumerate}
\item [(i)] the angle between geodesics $V_{i+1}P_{i+2}$ and $V_{i+1}Q_{i+1}$ is greater than $\frac{\pi}{4}$,
\item [(ii)] the angle between geodesics $V_{i}Q_{i-1}$ and $V_{i}P_i$ is greater than $\frac{\pi}{4}$.
\end{enumerate}
\end{lem}
\begin{proof}
Let the Euclidean distance from the center of the unit disk $\D$, $O$, to the center of each isometric circle be $d$, the Euclidean radius of each isometric circle by $R$, and $v$ be the distance from $O$ to the vertex $V_{i+1}$
(see Figure \ref{anglecalc}). 
The angle between the imaginary axis and the ray from the origin to $V_{i+1}$ is equal to $t=\frac{\pi}{8g-4}$.  
The angle between geodesics $V_{i+1}P_{i+2}$ and $V_{i+1}Q_{i+1}$ is equal to the angle between the radii of the  Euclidean circles (of centers $O_i$, $O'_{i+1}$) representing these geodesics,  i.e., $\angle O_iV_{i+1}O'_{i+1}$.  Our goal is to express it as a function of $t$, $\omega(t)$.

Let $\varphi=\angle O_iOQ_{i+1}$. 
We have $\sin\varphi=|O_iQ_{i+1}|/d$, and $\sin t=|O_iH|/d$, where $O_iH\perp OH$.
Since the angle of $\F$ at $V_{i+1}$ is equal to $\frac{\pi}2$, $|O_iH|=|O_iV_{i+1}|/\sqrt{2}$, and since 
$|O_iV_{i+1}|=|O_iQ_{i+1}|=R$, we obtain
\begin{equation}
\sin\varphi=\sqrt{2}\sin t,
\end{equation}
and therefore
\begin{equation}
\cos\varphi=\sqrt{\cos(2t)}.
\end{equation}

\begin{figure}[tb]
\includegraphics[scale=0.72]{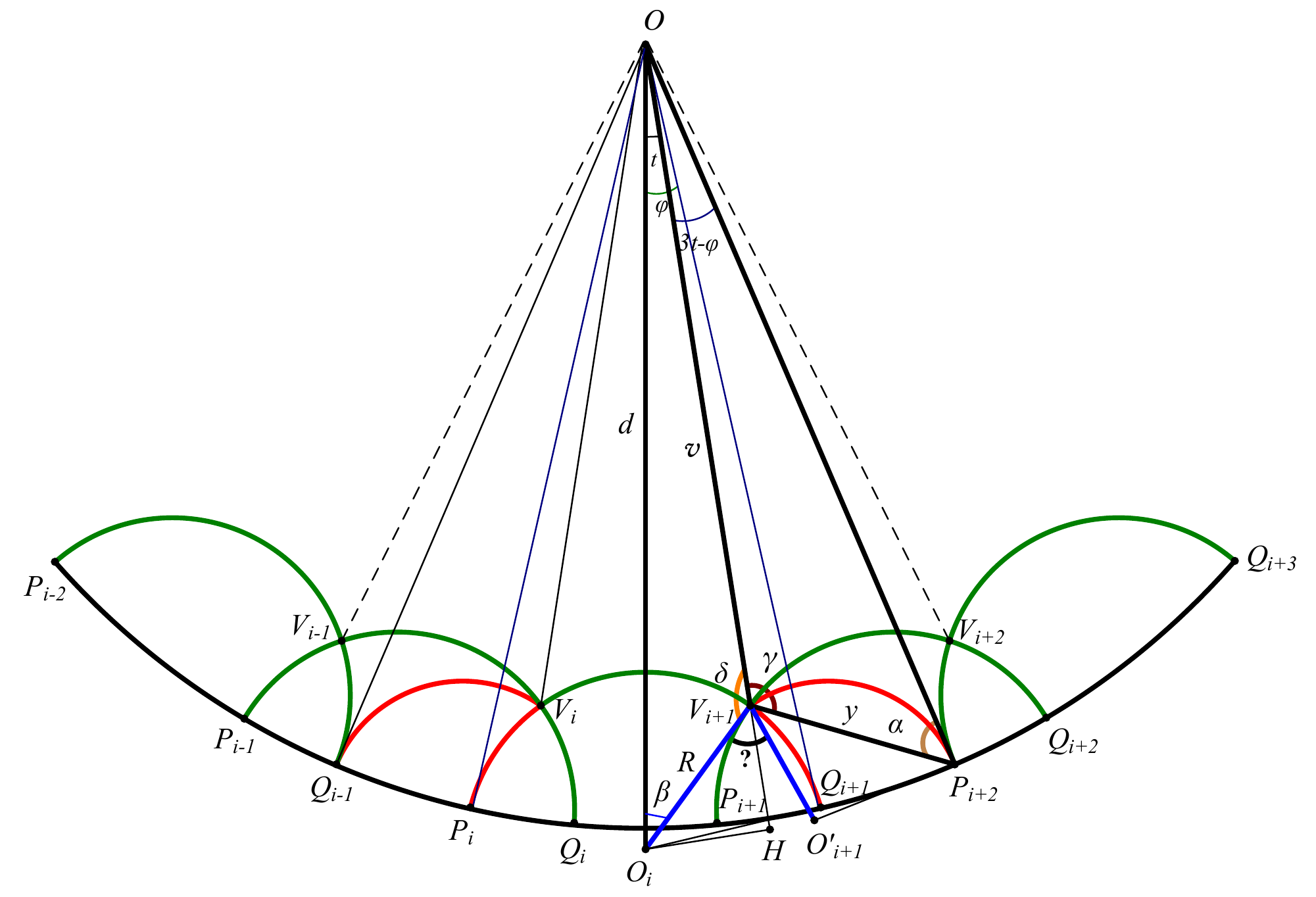}
\caption{Calculation of angle $\angle O_iV_{i+1}O'_{i+1}$}
\label{anglecalc}
\end{figure}

In the right triangle $\Delta O_iOH$ we have $|OH|=v+\frac{R}{\sqrt{2}}$ and $|O_iH|=\frac{R}{\sqrt{2}}$, hence by the Pythagorean Theorem,
\[
\left(v+\frac{R}{\sqrt{2}}\right)^2+\frac{R^2}2=d^2=\frac{R^2}{2\sin^2t},
\]
which implies
\[
v+\frac{R}{\sqrt{2}}=\frac{R}{\sqrt{2}}\cot t,
\]
and hence
\[
v=\frac{R}{\sqrt{2}}\left(\frac{\cos t}{\sin t}-1\right).
\]
Using that $R=\sqrt{2}\sin(t)d$ and $d=\displaystyle\frac{1}{\cos\varphi}=\frac{1}{\sqrt{\cos(2t)}}$, we obtain $R$ and $v$ as functions of $t$,
\begin{equation}
R(t)=\frac{\sqrt{2}\sin t}{\sqrt{\cos(2t)}}, \quad v(t)=\sqrt{\frac{\cos t-\sin t}{\cos t+\sin t}},
\end{equation}
and we now can express all further quantities as functions of $t$.

In the triangle $\Delta OO_iV_{i+1}$, let $\angle OO_iV_{i+1}=\beta(t)$ and $\angle OV_{i+1}O_i=\delta(t)$.
 In the triangle $\Delta OP_{i+2}V_{i+1}$, let $|V_{i+1}P_{i+2}|=y(t)$, $\angle OP_{i+2}V_{i+1}=\alpha(t)$, $\angle OV_{i+1}P_{i+2}=\g(t)$. One can easily see that $\angle V_{i+1}OP_{i+2}=3t-\varphi(t)$. 
Using the Rule of Cosines, we have
\[
y(t)^2=1+v(t)^2-2v(t)\cos(3t-\varphi).
\]
Using the Rule of Sines in the  triangles $\Delta OP_{i+2}V_{i+1}$ and $\Delta OO_iV_{i+1}$ we obtain
\[
\sin(\alpha(t))=\frac{v(t)\sin(3t-\varphi)}{y(t)},\quad
\sin(\beta(t))=\frac{v(t)\sin(t)}{R(t)}=\frac{\cos t-\sin t}{\sqrt{2}},
\]
and the last equation implies $\beta=\frac{\pi}{4}-t$.

The angle  $\omega(t)=\angle O_{i}V_{i+1}O'_{i+1}$ in question is calculated as
\[
\omega(t)=2\pi-\gamma(t)-\delta(t)-\left(\frac{\pi}{2}-\alpha(t)\right).
\]
Expressing $\gamma(t)$ and $\delta(t)$ from these triangles we obtain
\begin{equation}
\begin{aligned}
\omega(t)&=4t-\varphi(t)+2\alpha(t)+\beta(t)-\frac{\pi}{2}\\
&=4t-\varphi(t)+2\alpha(t)+\frac{\pi}{4}-t-\frac{\pi}{2}\\
&=3t-\varphi(t)+2\alpha(t)-\frac{\pi}{4}.
\end{aligned}
\end{equation}
We see that the desired inequality
\begin{equation}\label{P>}
\omega(t)>\frac{\pi}{4}
\end{equation} is equivalent to $3t-\varphi(t)+2\alpha(t)>\frac{\pi}{2}$, and since
from $\Delta OV_{i+1}P_{i+2}$ we have
\[
3t-\varphi(t)+\alpha(t)+\g(t)=\pi,
\]
 (\ref{P>}) is equivalent to 
 \begin{equation}\label{final}
 \g(t)-\alpha(t)<\frac{\pi}{2}.
 \end{equation}
 Recall that $\g(t)$ and $\alpha(t)$ are the angles of the triangle $\Delta OV_{i+1}P_{i+2}$, with $\g(t)> \frac{\pi}{2}$ and $\alpha(t)<\frac{\pi}{2}$, hence $0<\g(t)-\alpha(t)<\pi$.
  In order to prove (\ref{final}), we need to show that 
   \begin{equation}\label{diff}
   \cos (\g(t)-\alpha(t))>0.
   \end{equation} 
 Using the Rule of Sines we obtain
 \[
 \sin\g(t)=\frac{\sin\alpha(t)}{v(t)}.
 \]
 Using the Rule of Cosines we obtain
 \[
 \cos\g(t)=\frac{y^2(t)+v^2(t)-1}{2y(t)v(t)}\text{ and }\cos\alpha=\frac{1+y^2(t)-v^2(t)}{2y(t)}.
 \]
 In what follows we will suppress dependence of all functions on $t$. Thus
 \[
 \begin{aligned}
  \cos (\g-\alpha)&=\cos\g\cos\alpha+\sin\g\sin\alpha\\
  &=\frac{(y^2+v^2-1)(1+y^2-v^2)}{4y^2v}+\frac{\sin^2\alpha}{v}\\
 &=\frac{8v^2 -4v(1+v^2)\cos(3t-\varphi)}{4vy^2}.
 \end{aligned}
 \]
 Since $v$ and $y$ are positive, it is sufficient to prove the positivity of the function
\[
\begin{aligned}
g(t)&=\frac{2v}{(1+v^2)}-\cos(3t-\varphi)=\frac{\cos\varphi}{\cos t}-\cos(3t-\varphi)\\
&=\frac{\cos\varphi}{\cos t}-\cos((3t-2\varphi)+\varphi)\\
&=\frac{\cos\varphi}{\cos t}-(\cos(3t-2\varphi)\cos\varphi-\sin(3t-2\varphi)\sin\varphi)\\
&=\cos\varphi\left(\frac1{\cos t}-\cos(3t-2\varphi)\right)+\sin(3t-2\varphi)\sin\varphi.
\end{aligned}
\]
The first term is positive since $\cos\varphi$, $\cos t$ and $\cos(3t-2\varphi)$ are less than $1$. The second term is positive since 
\begin{equation} \label{positive}
3t-2\varphi>0.
\end{equation} 
The latter follows from the  fact that the function
$$h(t)=3t-2\varphi(t)=3t-2\arcsin(\sqrt{2}\sin t)$$
has second derivative
$$
h''(t)=-\frac{2\sqrt{2}\sin t}{\cos^{3/2}(2t)}
$$
negative on $(0,\pi/12]$, hence
$$h'(t)=3-\frac{2\sqrt{2}\cos t}{\cos^{1/2}(2t)}$$
is  decreasing on $(0,\pi/12]$, so $h'(t)\ge h'(\pi/12)=3-\displaystyle\frac{\sqrt{2}+\sqrt{6}}{{3}^{1/4}}>0$ for any $t\in (0,\pi/12]$. Thus, $h$ is strictly increasing on $(0,\pi/12]$, so $h(t)>h(0)=0$ for any $t\in (0,\pi/12]$ which implies \eqref{positive}.
Thus (\ref{P>}) follows.
The second inequality follows from the symmetry of the fundamental domain $\F$.
\end{proof}
In what follows $\ell$ will denote the arc length on the unit circle $\Sb$.
\begin{cor}\label{new}$ $
\begin{enumerate}
\item [(i)]There exist $a_j, b_j\in(P_j,Q_j)$ such that $d(P_j,a_j)>\frac12\ell(P_j,Q_j)$ and $\ell(b_j,Q_j)>\frac12 \ell(P_j,Q_j)$ such that $T_j(a_j)=P_{\rho(j)+1}$ and $T_{j-1}(b_j)=Q_{\theta(j-1)}$.
\item [(ii)] For any point $x\in [P_j,Q_j]$ such that $\ell(P_j,x)\le\frac{1}{2}\ell(P_j,Q_j)$, we have $T_j(x)\in [Q_{\sigma(j)+1},P_{\sigma(j)+2}]$. 
\item [(iii)] For any point $x\in [P_j,Q_j]$ such that $\ell(x,Q_j)\le\frac{1}{2}\ell(P_j,Q_j)$, we have
$T_{j-1}(x)\in [Q_{\theta(j-1)},P_{\theta(j-1)+1}]$.
\end{enumerate}
\end{cor}
\begin{proof} (i) Let $M_j$ be the midpoint of $[P_j,Q_j]$. Since the angle at each $V_j$ is equal to $\frac{\pi}2$, the angle between the geodesic segments $V_jP_j$ and $V_jM_j$ is equal $\frac{\pi}4$. Recall that $T_j([P_j,Q_j])=[Q_{\rho(j)},Q_{\rho(j)+1}]$. Since, by Lemma \ref{geometric} (i) for $i=\sigma(j)$,
 the angle between the geodesic segments $V_{\rho(j)}P_{\rho(j)+1}$ and $V_{\rho(j)}Q_{\rho(j)}$ is $>\frac{\pi}4$, and  $T_j$ is conformal, the existence of $a_j\in(M_j, Q_j)$ such that $T_j(a_j)=P_{\rho(j)+1}$ follows. Similarly, we know that $T_{j-1}([P_j,Q_j])=[P_{\theta(j-1)},P_{\theta(j-1)+1}]$. Since by  Lemma \ref{geometric} (ii) with $i=\sigma(j-1)$, the angle between the geodesic segments $V_{\sigma(j-1)}Q_{\theta(j-1)}$ and $V_{\sigma(j-1)}P_{\theta(j-1)+1}$ is greater than $\frac{\pi}{4}$ and $T_{j-1}$ is conformal, the existence of $b_j\in(P_j,M_j) $ such that $T_{j-1}(b_j)=Q_{\theta(j-1)}$ follows. Parts (ii) and (iii) follow immediately from (i).
\end{proof}

\begin{cor}\label{lem:app}
The arc length of the interval $T_k([P_{k+2},Q_{k+2}])$ is less than \linebreak$\frac 12$ of $[P_{\sigma(k)},Q_{\sigma(k)}]$ and the length of the interval $T_k([P_{k-1},Q_{k-1}])$ is less than $\frac12$ of $[P_{\sigma(k)+1},Q_{\sigma(k)+1}]$.
\end{cor}
\begin{proof}
By Proposition \ref{rem1}, we have $T_k(Q_{k+1})=P_{\sigma(k)}$ and $T_k(Q_{k+2})=Q_{\sigma(k)}$. The fact that  the length of $T_k([P_{k+2},Q_{k+2}])< \frac12 \ell(P_{\sigma(k)},Q_{\sigma(k)})$ is equivalent to the fact that $T_k(P_{k+2})\in[M_{\sigma(k)},Q_{\sigma(k)}]$, where $M_{\sigma(k)}$ is the middle of 
$[P_{\sigma(k)},Q_{\sigma(k)}]$. But the last statement follows from the fact that the angle between the geodesic $V_{k+1} P_{k+2}$ and the geodesic $V_{k+1}Q_{k+2}$ is less then $\frac{\pi}{4}$, a direct consequence of the fact that the angle in the part (i) of Lemma \ref{geometric} is greater that $\frac{\pi}{4}$. The second statement follows immediately  from the part (ii) of Lemma \ref{geometric}.
\end{proof}

\medskip

\noindent\textbf{Acknowledgments.}
We thank the anonymous referee for several useful comments and suggestions.


\begin{thebibliography}{99}

\bibitem{AF3} R. Adler, L. Flatto, \emph{Geodesic flows, interval maps, and
symbolic dynamics}, Bull. Amer. Math. Soc. {\bf 25} (1991), no.~2, 229--334.

\bibitem{AK} W. Ambrose, S. Kakutani, \emph{Structure and continuity of measurable flows}, Duke Math. J., {\bf 9} (1942), 25--42.

\bibitem{BoS} R. Bowen, C. Series, \emph{Markov maps associated with Fuchsian groups}, Inst. Hautes \'Etudes Sci. Publ. Math. No. 50 (1979), 153--170.

\bibitem{JS} G.A. Jones, D. Singerman, \emph{Complex Functions}, Cambridge University Press, 1987

\bibitem{K2} S. Katok, {\em Fuchsian Groups}, University of Chicago Press, 1992.

\bibitem{KU3} S. Katok, I. Ugarcovici, {\em Structure of attractors for $(a,b)$-continued fraction transformations}, Journal of Modern Dynamics, {\bf 4} (2010), 637--691. 

\bibitem{KU6} S. Katok,  I. Ugarcovici, \textit {Applications of $(a,b)$-continued fraction transformations}, ETDS (2012).

\bibitem{M} B. Maskit, On Poincar\'eÕs Theorem for fundamental polygons, Adv. Math.,
{\bf 7} (1971), 219--230.

\bibitem{S2} C. Series, {\em Symbolic dynamics for geodesic flows},
Acta Math.  {\bf 146} (1981), 103--128.

\bibitem{Sm} I. Smeets, {\em On continued fraction algorithms}, PhD thesis, Leiden, 2010.

\bibitem{Z} D. Zagier, {\em Possible notions of ``good'' reduction algorithms}, personal communication, 2007.

\end{thebibliography}
\end{document}